\documentclass[10pt,reqno]{amsart}
\usepackage{amsmath, amssymb, amsthm}
\usepackage{graphicx}
\usepackage{xcolor}
\usepackage{url}
\usepackage[breaklinks]{hyperref}
    
\usepackage{subcaption}
\newcommand{\Z}{\mathbb Z}

\newcommand{\SL}{\operatorname{SL}}

\renewcommand{\Re}{\operatorname{Re}}

\renewcommand{\(}{\left\(}
\renewcommand{\)}{\right\)}
\renewcommand{\[}{\left\[}
\renewcommand{\]}{\right\]}
\numberwithin{equation}{section}
 \theoremstyle{plain}
\newtheorem{theorem}{Theorem}[section]
\newtheorem{lemma}[theorem]{Lemma}
\newtheorem*{proposition*}{Proposition}
\newtheorem*{theorem*}{Theorem}

\newtheorem{defn}[theorem]{Definition}
\newtheorem{corollary}[theorem]{Corollary}
\newtheorem{remark}[theorem]{Remark}
\newtheorem{proposition}[theorem]{Proposition}

   \makeatletter
\def\proof{\@ifnextchar[{\@oproof}{\@nproof}}
\def\@oproof[#1][#2]{\trivlist\item[\hskip\labelsep\textit{#2 Proof of\
#1.}~]\ignorespaces}
\def\@nproof{\trivlist\item[\hskip\labelsep\textit{Proof.}~]\ignorespaces}
\makeatother

\title[Summation formulas for mock modular forms]{Summation formulas for Hurwitz class numbers and other mock modular coefficients}

\author{Olivia Beckwith}\address{Department of Mathematics,
Tulane University, New Orleans, LA 70118}
\email{obeckwith@tulane.edu}
\author{Nikolaos Diamantis}
\address{Room C22 The Mathematical Sciences Building, University Park,
Nottingham,
NG7 2RD,
UK }
\email{nikolaos.diamantis@nottingham.ac.uk}
\author{Rajat Gupta}
\address{University of Maine, USA,
Room 321, Neville Hall}
\email{rajat.gupta@maine.edu}
\author{Larry Rolen}
\address{Department of Mathematics,
1420 Stevenson Center,
Vanderbilt University,
Nashville, TN 37240 }
\email{larry.rolen@vanderbilt.edu }
\author{Kalani Thalagoda}\address{Department of Mathematics,
Tulane University, New Orleans, LA 70118}
\email{kthalagoda@tulane.edu}

\date{\today}

\begin{document}

\begin{abstract}
We prove a formula for weighted sums of the first $n$ coefficients of mock modular forms of moderate growth and apply it to Hurwitz class numbers and coefficients of negative half integral weight Eisenstein series, which take the form of certain quadratic Dirichlet $L$-values. Our formula is a mock modular version of a Bessel-sum identity proved by Chandrasekharan and Narasimhan for Dirichlet series satisfying a functional equation. Our proof utilizes $L$-functions for mock modular Eisenstein series defined by Shankadhar and Singh. 
\end{abstract}

\maketitle

\section{Introduction}
In this paper, we develop a new theory of summation formulas for a class of distinguished mock modular forms. To do so, we use ideas inspired by insights from an old paper of Chandrasekharan and Narasimhan \cite{ChandNaras61} and combine them with the modern toolkit of harmonic Maass forms. We will describe the general framework in Section \ref{Section1.2}. Before this, we begin by describing two key applications of the general results which formed the primary motivation for writing this paper. 

\subsection{Motivating Example: Class Numbers}
Class numbers of imaginary quadratic fields have long been of central interest in number theory. First studied by Gauss through the lens of binary quadratic forms under the action of $\mathrm{SL}_2(\mathbb Z)$, the class number $h(-D)$
%\textcolor{purple}{for $D>0$}
%(see Section \ref{} for the definitions and basic results on class numbers) 
quantifies the failure of the rings of integers in $\mathbb Q(\sqrt{-D})$ to have unique factorization. Gauss conjectured that $h(-D)\rightarrow \infty$ as $D\rightarrow \infty$. 
%In particular, he conjectured that are only finitely many discriminants $-D<0$ with $h(-D)=1$, the largest being $-D=-163$. This last fact and this was finally proven by Heegner and Stark about 60 years ago \cite{Heegner,Stark}. 
Gauss's conjecture was proven in general by Heilbronn \cite{Heilbronn}, which Watkins used to compute exact lists of discriminants with class numbers up to 100 \cite{Watkins}.

As a sequence, 
%of numbers
the class numbers $h(D)$ and variations of them, such as the Hurwitz class numbers $H(n)$ (Sect. \ref{Hurwitz}) have many interesting arithmetic and analytic properties.  For example, they satisfy many beautiful recurrence formulas such as that of Hurwitz and Kronecker:
\begin{equation}
\label{HurwitzKronecker}
\sum_{r\in\mathbb Z}H(4n-r^2)+\lambda_1(n)=2\sigma_1(n).
\end{equation}
Here 
%$H(D)$ is the 
%slightly renormalized 
%Hurwitz class number (Sect. \ref{Hurwitz}), 
$\lambda_1(n):=\sum_{d|n}\min(d,n/d)$, and $\sigma_k(n):=\sum_{d|n}d^k$.

A natural question, of central interest in analytic number theory, is to ask for good bounds on class numbers. That is, how precise can one make Heilbronn's theorem? In \cite{Siegel}, Siegel gave a celebrated first answer  to this question by proving the ineffective bound
%. He proved that for any $\varepsilon>0$, there is a constant $c(\varepsilon)>0$ such that 
$
h(-D)\geq c(\varepsilon) D^{\frac12-\varepsilon}$ (for $D >0$ sufficiently large).
%Siegel proved this by relating the problem to the Generalized Riemann Hypothesis, but it is completely ineffective in that it doesn't give control over estimating the constant $c(\varepsilon)$. 
On the other hand, 
Goldfeld, Gross, and Zagier \cite{GoldfeldGrossZagier} proved the effective bound:
%the best current effective lower bound is due to Goldfeld, Gross, and Zagier \cite{GoldfeldGrossZagier}:
$$
h(-D) > \frac{1}{7000}\left(\log D\right) \prod_{\substack{p\mid D\\ p\neq D}} \left(
1-\frac{[2\sqrt{p}]}{p+1}\right).
$$
As one can see, there is a large gap between what is asymptotically true and what can be proven effectively. Thus, as is typical in analytic number theory, many papers have studied asymptotics of sums and moments of class numbers. 
For example, Wolke \cite{Wolke} proved that for all $\alpha, \varepsilon>0,$ there is a constant $c(\alpha)$ such that
$$
\sum_{n \le X} h(-n)^{\alpha} = c(\alpha) X^{1 + \frac{\alpha}{2}} + O_{\alpha}(X^{\frac{3}{4} + \frac{\alpha}{2}+\varepsilon}).
$$
Another important recent result is due to Walker \cite{Walker-2024}, who proved that
$$
\sum_{n\leq X}H(n)H(n+\ell)=\frac{\pi^2X^2}{252\zeta(3)}\left(2\sigma_{-2}(\ell/4)-\sigma_{-2}(\ell/2)+\sigma_{-2}(\ell_o)\right)+O\left(X^{\frac53+\varepsilon}+X^{1+\varepsilon}\ell\right),
$$
for any $\varepsilon>0$ $\ell\in \mathbb N$, and where $\ell_o$ denotes the odd part of $\ell$. To give one more example, Ono, Saad, and Saikia \cite{OnoSaadSaikia} used formulas and asymptotic for moments of the form
$
\sum_{s} H\left(\frac{n - 4s^2}{4}\right) s^{2m}
$
for non-negative integers $m$
to study the distribution of values of Gaussian hypergeometric functions.

This paper proves new analytic estimates for sums of class numbers by studying them from a mock modular standpoint.
%, and more generally, for a special class of sequences tied to mock modular forms. 
%When a sequence of numbers is related to a modular object in some way, this modularity gives access to a powerful tooblox to prove analytic and arithmetic properties.
A hint of the modularity properties of $H(n)$ is provided by \eqref{HurwitzKronecker}; the right-hand side is essentially the $n$-th Fourier coefficient of the weight $2$ (quasi)modular Eisenstein series $E_2(\tau)$
%. Moreover, the first term 
while, in the left-hand side, we recognize a convolution of the generating function $\mathcal H^+(\tau):=\sum_{D\geq0}H(D)$ with the weight $1/2$ Jacobi theta function $\vartheta(\tau):=\sum_{n\in\Z}q^{n^2}$.
%This is indeed the case. 
However, the modular objects involved are more complicated than classical holomorphic modular forms:
%. Exactly 
50 years ago, Zagier \cite{zagier75} proved that $\mathcal H^+(\tau)$  can be ``completed'' by adding a non-holomorphic integral of $\vartheta(\tau)$ to %give 
a function which transforms as a weight $3/2$ modular form. 
%of weight $1/2$
In modern language, 
%this says that 
the Hurwitz class number generating function is a {\it mock modular form}, and it can be completed to give a {\it harmonic Maass form} (see Sect. \ref{PolyGrowthHMFs}).

%This realization allows one to prove many properties of class numbers. 
%Indeed, \eqref{HurwitzKronecker} can be proven directly by using Sturm's theory of holomorphic projection. Many further relations were proven by Mertens in \cite{Mertens}, settling a conjecture of Cohen.

%In this paper, we develop new tools 
Here, we exploit this interpretation to study sums of coefficients of distinguished mock modular forms such as $\mathcal H^+(\tau)$. In particular, our main applications to class numbers can be found below in Theorems~\ref{Prof82} and \ref{applHur}.
For example, Theorem~\ref{applHur} states the following.
\begin{theorem*} \label{thm:intro} Let $H(n)$ denote the Hurwitz class number.
%and let $g_{\rho}(n,x)$ be as given in theorem \eqref{thm:perrongen}. 
For $\rho>1$ and $\epsilon > 0$, we have
\begin{align}\label{IntroAs}
&\frac{1}{\Gamma(\rho+1)}\sum_{n\leq x}H(n)(x-n)^{\rho} 
%+\frac{x^{\rho}}{8 \pi^{\frac{3}{2}} i } \sum_{n \le \sqrt{x}} n g_{\rho} \left(n^2, x \right)
=\frac{ \pi^{\frac{3}{2}} x^{\rho + \frac{3}{2}}  }{24 \Gamma \left (\rho + \frac{5}{2} \right )}
+ O(x^{\rho+1 + \epsilon}).
\end{align}    
\end{theorem*}
This theorem corresponds to the case of weight $3/2$, and, in Section \ref{neg1/2}, we prove its counterpart for {\it positive} half-integral weights (Theorem \ref{th:AsymShimEis}). Instead of Hurwitz class numbers, this involves explicit products of special vales of Dirichlet $L$-functions, which can be thought of as versions of the generalized Hurwitz class numbers (cf. \cite{beckwith-mono}, \cite{mono}).

In the next section, we will give an overview of the more general version of our results, and of the context within the theory of mock modular forms.

\subsection{Summation formulas for mock modular forms}\label{Section1.2}
The purpose of this paper is to prove a new class of ``summation formulas,'' from which the last theorem follows. 
%Given the importance of class numbers, and as they are coefficients of a near-modular object, it is natural to ask how one can bootstrap this (mock) modularity to estimate them. This paper does so.
%To describe this, we first describe the context of such summation formulas. 
Summation formulas are now essential tools in analytic number theory (see \cite{MS2004} for a beautiful survey of their history and connection to modularity).

%can be found in \cite{MS2004}.

An early example of such a formula is Voronoi's enhancement of Dirichlet's estimate for the number of divisors $d(n):=\sigma_0(n)$
%Voronoi pioneered this subject 
%\cite{givepapers27-29fromMillerSchmid} 
%developing a estimates for Dirichlet's classical {\it divisor problem}. This seeks to estimate the number of divisors function %$d(n):=\sigma_0(n)$. Dirichlet proved the estimate
%\begin{equation}\label{Dirichlet}
%\sum_{n=1}^xd(n)=x\log x-(2\gamma-1)x+O(x^{1/2}),
%\end{equation}
%where $\gamma$ is the usual Euler-Mascheroni constant. Voronoi's results allowed one to study such sums 
with the inclusion of {\it test functions} $f(x)$:
%By varying such weightings and sampling carefully, one can greatly improve estimates such as Dirichlet's. Specifically, Voronoi proved that if $f(x)$ is any piecewise continuous, piecewise monotone function, and if $Y_0$ and $K_0$ denote the standard Bessel functions, then 
\begin{equation}
\label{Voronoi}
\sum_{a\leq n\leq b}d(n)f(n)=\int_a^bf(x)(\log x+2\gamma)dx+\sum_{n\geq1}
d(n)\int_a^bf(x)\left(4K_0(4\pi \sqrt{nx})-2\pi Y_0(4\pi\sqrt{nx})\right)dx.
\end{equation}
Here, $\gamma$ is the Euler-Mascheroni constant, $Y_0$ and $K_0$ are the standard Bessel functions and the summation on the left-hand side is understood to be 
%denotes that $f(n)$ should be interpreted as 
an average of the left and right-handed limits at points where $f(n)$ is discontinuous.
The added flexibility allowed by the test functions in this summation formula leads to
%This can be used to  give a 
power savings over Dirichlet's estimate.
%in \eqref{Dirichlet}. In particular, one can reduce the error to $O(x^{1/3})$.

Since then, summation formulas have been the focus of intense research from a variety of viewpoints one of the most important of which 
is based on the language of modular forms. However, as yet, they have not played a significant role in the theory of { \it mock} modular forms and harmonic Maass forms, and, here, we aim to rectify this. 
%Formulas like \eqref{Voronoi} are known as summation formulas, and are now well-understood in general contexts. In particular, divisor sums like $d(n)$ are the Fourier coefficients of (Maass) Eisenstein series, and similar results can be proven for holomorphic modular forms and Maass waveforms. As described in \cite{MS2004}, this machinery has even been extended to automorphic forms for higher $\mathrm{GL}(n)$. This machinery runs through the language of automorphic $L$-functions.
%This paper develops similar techniques for coefficients of so-called mock modular forms. 
The systematic theory of mock modular forms was initiated by 
%se first arose from a seminal thesis of 
Zwegers \cite{zwegers}, who solved an 80-year-old mystery about the modularity of Ramanujan's mock theta functions. Important milestones in the development of the subject included Bruinier--Funke's  general theory of harmonic Maass forms \cite{brufu02}, Bringmann--Ono's application of the theory for the proof of the Andrews–-Dragonette conjecture \cite{BringmannOno}, Bruinier--Ono's applications to the vanishing of central $L$-values and $L$-derivatives of rational elliptic curves \cite{BruO10}, Duke--Imamo\u{g}lu--T\'oth's deep
connections to real quadratic class numbers \cite{DIT} etc.

Although these forms have many of the features of classical holomorphic and Maass wave forms such as
%. For example, one has 
a robust theory of differential and Hecke operators,
%on which the spaces, and many of the core applications of the theory rely,
%on the properties of these. However, until recently, the study of harmonic Maass forms lacked 
some of the key features of classical modular forms were missing until recently. This included a systematic theory of $L$-series, which was significant because in the classical setting such $L$-series underlie most of the fundamental results. 
%For instance, the development of $L$-functions for these forms has only emerged in the last several years. One of the main reasons for this was that most harmonic Maass forms, if they are not classical holomorphic forms, necessarily have exponential growth at the cusps, and hence subexponential growth of their Fourier coefficients. This rapid growth is essential for applications in combinatorics and mathematical physics. On the other hand, it makes the definition of suitable $L$-functions much less clear. 
The first versions of $L$-series for these forms were given by Bringmann,
Fricke, and Kent \cite{BFK14} (in the weakly holomorphic case) and 
%. They provided $L$-functions for those harmonic Maass forms which are holomorphic on $\mathbb H$: the weakly holomorphic modular forms. While this allowed them to produce important extensions of classical Eichler-Shimura theory, it had two primary downsides. Namely, it was not clear how to extend the definition to the more general harmonic Maass forms, and it was defined in a way that ``baked in'' the functional equations of the $L$-functions. This made it impossible to move back and forth between the $L$-functions and the modular sides, as in Weil's Converse Theorem for the classical holomorphic modular form case. In particular, this excluded developments such as the study of summation formulas, which are intimately tied with results like the Converse Theorem.
%The first definition for non-holomorphic harmonic Maass forms was given 
by Shankhadhar and Singh \cite{ShanSingh22} (for some non-holomorphic harmonic Maass forms). 
%This extended the definition to those harmonic Maass forms which have no exponential growth at the cusps. While these are rare (mainly consisting of the class number generating function and the Cohen-Eisenstein series), this gave a first decoupled version without a functional equation baked in. 
The first completely general definition of $L$-series for harmonic Maass forms was given by Lee, Raji, and two of the authors in \cite{DLRR}. The approach adopted there used test functions and allowed for several of the familiar features of $L$-series such as  %There, they circumvented the exponential growth by using a distributional-style definition of $L$-functions as functions on a space of test functions. This is analogous to Miller-Schmid's theory of automorphic distributions \cite{MS2004}. The authors of \cite{DLRR} then used this new structure to prove 
functional equations, Weil-type converse theorems for harmonic Maass forms and a first instance of a summation formula. 
%This automatically produced summation formulas for coefficients of harmonic Maass forms.  These summation formulas are not the unique ones possible to prove. The goal of this paper is to provide summation formulas which are well-suited to detailed analytic analysis. 
%\subsection{Statement of results}

The aim of this work is to study specifically the last aspect, namely summation formulas, with a view towards arithmetic applications  of the same nature as those derived by their classical counterparts outlined in the last section. To this end, we focus on the simplest case of 
%harmonic Maass forms without exponential growth at the cusps, which we call 
{\it harmonic Maass forms of polynomial growth} (see Definition \ref{def:hmfpg}). On the one hand, this space is small enough to allow for a clear understanding of the mechanism of summation formulas and, on the other, it includes objects of clear arithmetic importance, such as the generating function of Hurwitz class numbers mentioned above. 

The main summation formula will be stated below in Theorem \ref{thm:summationformula}, but here we give the following specialization to the Hurwitz class numbers case that implies \eqref{IntroAs}:
\begin{theorem*} Let $H(n)$ denote the Hurwitz class number, $\chi_{-n}$ the character associated with $\mathbb Q(\sqrt{-n})$ and let $b^+(n)$ be a certain explicit linear combination of $H(n)$ and $L(\chi_{-n}, 1)$
%the coefficients of the holomorphic part of $\mathcal{H}(z) | w_4$ 
(see Theorem \ref{Prof82}). Then, for $\rho>0$. 
\begin{align*}
&\frac{1}{\Gamma(\rho+1)}\sum_{n\leq x}H(n)(x-n)^{\rho} - \frac{1}{4  } \sum_{n \le \sqrt{x}} n (x+ n^2)^{\rho} \left( \frac{2}{\Gamma(\rho+1)} + \frac{1}{ \sqrt{\pi} \Gamma(\rho + \frac{3}{2} )} B \left(  \frac{2}{x+1}, - \frac{1}{2}, \rho + \frac{3}{2} \right) \right) \nonumber
\\&-x^{\rho} \left ( \frac{1}{12 \Gamma(\rho+1)}+\frac{3xi^{\frac{3}{2}}(1+i)}{16 \Gamma(\rho+2)}-\frac{\sqrt{2 \pi x}}{8 \pi \Gamma(\rho+3/2)}-\frac{(1+i)(\pi i x)^{\frac{3}{2}}}{12\sqrt{8} \Gamma(\rho+5/2)}\right ) \nonumber \\
&=  \frac{-i^{\frac{3}{2}} (1+i)x^{\frac{2\rho+3}{4}}}{2 \sqrt{2}\pi^{\rho}}\sum_{n \ge 1} \frac{b^+(n)}{n^{\frac{2\rho+3}{2}}}J_{\rho+\frac{3}{2}}\left(2\pi\sqrt{nx}\right) \nonumber 
\\
&-
\frac{i^{\frac{3}{2}}(1+i)x^{\frac{\rho+1}{2}}}{4 \sqrt{2}\pi^{\rho+1}}\sum_{n \ge 1}\frac{1}{n^{\rho+1}}
\int_{0}^{1/2}\frac{u^{\frac{\rho}{2}}}{(1-u)^{\frac{1+\rho}{2}}
}J_{\rho+1}\left(2\pi n \sqrt{\frac{x(1-u)}{u}}\right)du
\end{align*}
where $B(a, b, c)$ is the incomplete Beta function and $J_\nu(x)$ the usual $J$-Bessel function.
\end{theorem*}

An important feature of our method is that it uses the old approach of \cite{ChandNaras61} on arithmetical identities derived from functional equations. We further use some of their results on trigonometric series to enlarge the range of validity of our summation formula in the important special case of weight $3/2$ (see Theorem \ref{thm:extendedsummation}).

The remainder of the paper is organized as follows. In Section~\ref{ComplexAnalysisBackground}, we establish the complex analytic 
prerequisites needed in the sequel, including new results for some key special functions. 
%such as a convenient version of the Phragm\'en-Linde\"of principle and facts about key special functions. Next, in 
In Section~\ref{PolyGrowthHMFs} we describe the basic facts and notation for harmonic Maass forms and their $L$-series. We then prove our main summation formula in Section~\ref{SummationSection}. In Section~\ref{extendingrange}, we extend this result to a larger range. These formulas are analyzed to give simpler asymptotic estimates for sums of coefficients of our harmonic Maass forms in Section~\ref{sec:AsymptoticAnalysis}. We apply this to the particular examples such as class numbers and Cohen--Eisenstein series in Section~\ref{application}. Finally, we conclude in Section~\ref{Converse} by showing that, in analogy with the results of Chandrasekharan--Narasimhan, our summation formulas imply functional equations, and hence modularity, of sequences of purported coefficients of harmonic Maass forms.

\section*{Acknowledgments}
The authors are grateful to Andreas Mono for many helpful comments on an earlier version of this manuscript. The first author was supported by NSF grant DMS-2401356 and Simons grant \#953473. The third author was supported by an AMS-Simons Travel Grant. The fourth author was supported by a grant from the Simons Foundation (853830, LR).

%%%%%%%%%%%%%

\section{Analytic preliminaries and some special functions}\label{ComplexAnalysisBackground}
\subsection{Preliminaries}
We fix the notation $z=x+iy$ for $z \in \mathbb C.$
We will use the following formulation of the Phragm\'en-Lindel\"of principle. 
\begin{proposition}\label{thm:pl} [\protect{\cite[\S 5.65]{Titch}}]
    Suppose that $f(z)$ is continuous in the two half-strips $S$ defined by $a \le x \le b$, $|y|>\eta$, analytic in the interior of $S$ and that, for all $\epsilon>0$, we have $f(z)=O(e^{\epsilon |y|})$ in $S$.    
    Then, if $f(a+iy)=O(|y|^{k_1})$ and $f(b+iy)=O(|y|^{k_2})$, for $|y|>\eta$, we have 
 $$f(x+iy)=O\left (|y|^{\frac{k_1-k_2}{a-b}x+\frac{k_2a-k_1b}{a-b}} \right), \qquad \text{for $|y|>\eta$,}$$ uniformly for $x \in [a, b]$.
\end{proposition}
 We will also use Stirling’s formula for the Gamma function.
\begin{proposition}\label{thm:stirling}[\protect{\cite[(5.11.9)]{DLMF}}]
For any $\sigma \in \mathbb{R}$, as $|t| \to \infty$ we have
\begin{align}\label{thm:st}
\Gamma(\sigma+it)\sim |t|^{\sigma-1/2}e^{-\pi|t|/2}
\end{align}
uniformly for bounded real values of $\sigma.$
\end{proposition}

Our method generalizes the method in \cite{ChandNaras61}, which is based on the following Perron formula.
\begin{proposition}\label{thm:perron}[\cite{ChandNaras61}, Lemma 1]
If  
$\varphi(s) = \sum_{n=1}^{\infty}a_n\lambda_n^{-s}$ with $\sum_{n=1}^{\infty}|a_n|\lambda_n^{-\alpha}<\infty$  for some $\alpha$ then for $\rho\geq 0, \sigma>0,$ and $\sigma\geq \alpha$ we have
$$ \frac{1}{\Gamma(\rho+1)} \sum{\vphantom{\sum}}'_{\lambda_n\leq x} a_n(x-\lambda_n)^{\rho} = \frac{1}{2\pi i}\int_{\sigma-i\infty}^{\sigma + i\infty} \frac{\Gamma(s)\varphi(s)x^{s+\rho}}{\Gamma(\rho+1+s)} ds,$$
where prime notation is indicating that the last term of the sum is multiplied by $1/2$ if $\rho=0$ and $x=\lambda_n$. 
\end{proposition}
Finally, we recall the asymptotics of the $J$-Bessel function.
\begin{proposition}\label{thm:besselj-asymptotics} [ 
(10.17.3)\cite{DLMF}]
For $\nu \in \mathbb C,$ set
    $(\nu,m) := \frac{\Gamma(\nu + m +1/2)}{\Gamma(m+1)\Gamma(\nu-m + 1/2)}$. Then, as $z \to \infty$ with $|\arg z| \le \pi-\delta$ (for some $\delta>0$), we have
    $$
        J_{\nu}(z)
        \sim 
        \left(\frac{2}{\pi z} \right)^{1/2} \left[\cos\left(z-\frac{\nu\pi}{2}-\frac{\pi}{4}\right)\sum_{m=0}^{\infty} \frac{(-)^{m}(\nu,2m)}{(2z)^{2m}} 
        - \sin\left(z-\frac{\nu\pi}{2}-\frac{\pi}{4}\right)\sum_{m=0}^{\infty} \frac{(-)^m(\nu,2m+1)}{(2z)^{2m + 1}}  \right].
    $$
\end{proposition}

\subsection{Special function arising from the non-holomorphic part}
The results in \cite{ShanSingh22} show that the $L$-function associated to a harmonic Maass form $f$ of polynomial growth involves the special function $W_\nu(s)$, for some $\nu \in \mathbb R$. This plays the role of the $\Gamma$-function for the contribution of the ``non-holomorphic part" of $f$ to the completed $L$-function of $f$. It is defined by the following formula for $\Re(s)>\max\{-\nu,0\}$:
\begin{equation}\label{eq:Wdef}
W_{\nu}(s):=\int_{0}^{\infty}\Gamma(\nu,2x)e^{x}x^{s-1}dx
\end{equation}
where $\Gamma(s,z)$ is the incomplete Gamma function, given, for $\operatorname{Re}(s) > 0$, by
\begin{align} \label{eq:incompleteGammaDef}
\Gamma(s,z) := \int_z^{\infty} t^{s-1} e^{-t} dt, \qquad z \in \mathbb{C}.
\end{align}
\begin{remark}
The function $W_{\nu}$ could be compared with $\mathbf \Gamma_s(y)$ of \cite{BDR} which can also be thought of as an iterated version of $\Gamma(s, z)$ and plays a key role in the sesquiharmonic Maass forms studied there.    
\end{remark}

To see that the integral in \eqref{eq:Wdef} converges for $\Re(s)>\max\{-\nu,0\}$, use the asymptotic behavior of $\Gamma(\nu,x)$ as $x\to 0$ (\cite{DLMF} 8.7.3) and as $x\to \infty$ (\cite{DLMF} 8.11.2). Note that, for $\nu \in -\mathbb{N}_0$, we draw this conclusion, through \cite{DLMF} 8.4.15:
$$\Gamma(\nu,z) = \frac{(-1)^n}{(-\nu)!} ( \psi(1-\nu) -\log z) + O(z^{\nu}) \qquad \text{if $\nu \in -\mathbb N_0$ as $z \to 0$}$$
where $\psi(z)$ is the Euler digamma function. We will now prove several properties of $W_{\nu}(s)$.

First we have the formula relating $W_{\nu}(s)$ to the Gamma function, the incomplete beta function $B(z, a, b)$ and Gauss's hypergeometric function $_2F_1(a, b, c; s)$:
%$$ B(z,a,b) := \int_0^z t^{a-1} (1-t)^{b-1} dt.$$
\begin{lemma}\label{thm:Wformula}
For $\Re(s)>\max\{-\nu,0\}$, we have the formulas
\begin{align}\label{W}
W_{\nu}(s)&=B\left(\frac{1}{2};s,1-s-\nu\right)\Gamma(s+\nu)
= \frac{\Gamma(s+\nu)}{2^s s}\, _2F_1 \left(s, s+\nu, 1+s; \frac{1}{2} \right).
\end{align}
\end{lemma}
\begin{proof}
We obtain the first equality using a straightforward calculation involving a series of $u$-substitutions:
	\begin{align*}
		W_{\nu}(s)&=\int_{0}^{\infty}\Gamma(\nu,2x)e^{x}x^{s-1}dx \\
  %&=\int_{0}^{\infty} \left(\int_{2x}^{\infty} e^{-u}u^{\nu-1}du\right) e^{x}x^{s-1}dx \\
%t = u/x -> u = tx -> du = xdt
&=\int_{0}^{\infty} \left(\int_{2}^{\infty} e^{-tx}t^{\nu-1}dt\right) e^{x}x^{s-1+\nu}dx\\
&= \int_{2}^{\infty} t^{\nu-1} \left(\int_{0}^{\infty} e^{(1-t)x}x^{s-1+\nu} dx\right) dt \\ % \tag{$w= x(t-1),dw = (t-1)dx$}
&= \int_{2}^{\infty} t^{\nu-1}(t-1)^{-s-\nu} \left(\int_{0}^{\infty} e^{-w}w^{s-1+\nu} dw\right) dt\\ 
%\tag{$w = 1/t, dt = -w^{-2} dw$}
%&=\Gamma(s+\nu) \int_{2}^{\infty} t^{\nu-1}(t-1)^{-s-\nu} dt\\ \nonumber
%&= \Gamma(s+\nu)\int_{0}^{1/2} w^{s-1}(1-w)^{-s-\nu} dw\\
&=\Gamma(s+\nu)B(1/2,s,1-s-\nu).
\end{align*}
We justify the integral swaps by the absolute convergence of the integral when Re$(s+\nu)>0$ and the application of Fubini's theorem.  We use \cite[Eq. 8.391]{table} to derive both the last line of the above calculation and the second formula in the lemma.  
\end{proof}
This lemma has a corollary, the second part of which is an analogue of Stirling's formula for $W_{\nu}.$
\begin{corollary}\label{thm:Wbound}
The function $W_{\nu}(s)$, defined by \eqref{eq:Wdef} for $\Re(s)>\max\{-\nu,0\}$, extends meromorphically to the entire complex plane with poles at $s \in S:=\{0,-1,-2,...\} \cup \{ -\nu,-\nu-1,-\nu-2,... \}$. The order of the pole at $s \in S$ is given by the multiplicity of $s$ in $S$ viewed as a multiset. Further, let $\nu,\alpha \in \mathbb{R}$, assume $\alpha > 0$. Then as $|t| \to \infty$, we have
$$
W_{\nu}(\alpha+it) \ll_{\alpha,\nu} |t|^{\alpha + \nu - \frac{1}{2}} e^{-\pi |t| /2}.
$$
\end{corollary}
\begin{proof} We consider the second equality of Lemma \ref{thm:Wformula}. Then the analytic continuations of the Gamma function and $_2F_1$ hypergeometric function give us the analytic continuation of $W_{\nu}(s)$. Further, the Gamma function has poles at $s+\nu\in -\mathbb{N}_{0}$ \cite[5.2.1]{NIST:DLMF} and $_2F_1$ hypergeometric function has poles at $1+s\in -\mathbb{N}_{0}$ \cite[15.2.1]{NIST:DLMF}. Therefore, $W_{\nu}(s)$ has poles at $s\in -\nu-\mathbb{N}_{0}$, $s\in -1-\mathbb{N}_{0}$ and at $s=0$. The asymptotic is immediate from Proposition~\ref{thm:st} and Lemma \ref{thm:Wformula}:
\begin{align*}
|W_{\nu}(\alpha+it)| \ll
\left (\int_0^{\frac{1}{2}} u^{\alpha - 1} (1-u)^{-\alpha - \nu} du \right ) |t|^{\alpha + \nu - \frac{1}{2}} e^{-\pi |t| /2}
\end{align*}
\end{proof}

We will use the following analogue of Proposition~\ref{thm:perron} for the $W_{\nu}(s)$ function. 
\begin{theorem}
\label{thm:perrongen}
Let $k, \rho, r, \alpha \in \mathbb{R}$ with $r, \alpha>0$. Assume that $\alpha + 1 -k > 0$, $k + \rho > 1$, and also that on the line $Re(s) = \alpha$, $L(s) = \sum_{n=1}^{\infty} b(n) n^{-s}$ is absolutely convergent. Then we have
\begin{align}\label{g(n, x)}
\sum_{n \le r} b(n) g_{\rho}(n,r) = \int_{(\alpha)} \frac{L(s) W_{1-k}(s)}{\Gamma(\rho + 1 + s)} r^s ds
\end{align}
where
$$g_{\rho}(n,r)=\frac{2 \pi i}{\Gamma(k+\rho)} \left (1+\frac{n}{r} \right )^{\rho}\int_{\frac{2n}{r+n}}^1 v^{-k} (1-v)^{k+\rho-1} dv.$$
When $k \not \in \mathbb{N}$, $g_{\rho}(n,r)$ equals
$$\frac{2 \pi i}{\Gamma(k+\rho)} \left (1+\frac{n}{r} \right )^{\rho}\left (\frac{\Gamma(1-k)\Gamma(\rho+k)}{\Gamma(\rho+1)}-B \left (\frac{2}{r+1}, 1-k, \rho+k \right ) \right ).
$$
\end{theorem}
\begin{proof}
From Stirling's formula (Proposition~\ref{thm:stirling}) and the previous result, we deduce that the integral $\int_{(\alpha)}  \frac{W_{1-k}(s)}{\Gamma(\rho + 1 + s)} y^s ds$ is absolutely convergent. Using the absolute convergence of $L(s)$ on $Re(s) = \alpha,$ $\int_{(\alpha)} \frac{L(s) W_{1-k}(s)}{\Gamma(\rho + 1 + s)} r^sds$ is also absolutely convergent, and we can swap the sum and integral as follows:
\begin{align*}
\int_{(\alpha)} \frac{L(s) W_{1-k}(s)}{\Gamma(\rho + 1 + s)} r^sds &= \sum_{n=1}^{\infty} b(n) \int_{(\alpha)}  \frac{W_{1-k}(s)}{\Gamma(\rho + 1 + s)} (r/n)^s ds.
\end{align*}
We evaluate the more general integral with $r/n$ replaced by any $y>1$. Using Lemma~\ref{thm:Wformula}, we have 
\begin{align*}
\int_{(\alpha)}  \frac{W_{1-k}(s)}{\Gamma(\rho + 1 + s)} y^s ds &
%= \int_{(\alpha)}  \frac{\Gamma( s + 1 -k) B(\frac{1}{2}, s, k-s)}{\Gamma(\rho + 1 + s)} y^s ds \\&
= \int_{(\alpha)}  \frac{\Gamma( s + 1 -k) \int_0^{\frac{1}{2}} u^{s-1} (1-u)^{k-s-1} du }{\Gamma(\rho + 1 + s)} y^s ds \\ 
&= \int_0^{\frac{1}{2}}  (1-u)^{k-1} u^{-1} \left(\int_{(\alpha)}  \frac{\Gamma( s + 1 -k)}{\Gamma(\rho + 1 + s)} \left(\frac{1-u}{uy}\right)^{-s} ds  \right) du.
\end{align*}
 By \cite[sec 7.3, eq. 20]{erd}, for $\gamma, a>0$, we have
\begin{align}\label{invMe}
\frac{1}{2\pi i}\int_{(\gamma)}  \frac{\Gamma(s)}{\Gamma(s + a)} y^{-s} ds &= \frac{(1-y)^{a-1}}{\Gamma(a)}\mathbf 1_{(0, 1)}(y),
\end{align}
where $\mathbf 1_{S}(y)$ is the indicator function of the set $S$. This means, that 
if $\alpha>k-1$ and $\rho +k>0$, then, with the change of variables $v=u(1+y)$, we get 
\begin{multline}\label{preform}
\frac{2 \pi i}{\Gamma(k+\rho)} \int_{\frac{1}{1+y}}^{\frac{1}{2}} \left(1 -  \left(\frac{1-u}{uy}\right) \right)^{-1+k+\rho}  \left(\frac{1-u}{uy}\right)^{1-k}  (1-u)^{k-1} u^{-1} du \\
= \frac{2 \pi i }{\Gamma(k+\rho)} \left ( \frac{1+y}{y} \right )^{\rho}\int_{1}^{\frac{y+1}{2}}  v^{-1-\rho} (v-1)^{-1+k+\rho} dv.
\end{multline} 
With the change of variables $t=1/v$, the integral becomes
$\int_{\frac{2}{y+1}}^1  v^{-k} (1-v)^{k+\rho-1} dv$
which, in turn, if $k \not \in \mathbb N$, equals
\begin{equation*}
 \left (\int_0^1-\int_0^{\frac{2}{y+1}} \right ) v^{-k} (1-v)^{\rho+k-1} dv=
B(1, 1-k, \rho+k)-B \left (\frac{2}{y+1}, 1-k, \rho+k \right ).
\end{equation*} 
From this, combined with \eqref{preform} and the formula 
\cite{DLMF} 5.12.1 expressing $B(1, a, b) = B(a,b)$ in terms of the Gamma function, we deduce the result. 
\end{proof}

Finally, we note that the function $W_\nu(s)$ satisfies a recursive formula analogous to the functional equation for $\Gamma(s)$.
\begin{lemma}
For any $\Re(s)>0$ and $\nu > 0$, we have
$$sW_{\nu}(s) = 2^{\nu}\Gamma(s+\nu)-W_{\nu}(s+1)$$
\end{lemma}
\begin{proof} Since, by (8.11.2) of \cite{NIST:DLMF}, $\Gamma(a, x) \ll e^{-x}x^{a-1}$, as $x \to \infty,$ integration by parts gives
    \begin{align*}
	sW_{\nu}(s)&=\int_{0}^{\infty}\Gamma(\nu,2x)e^{x}sx^{s-1}dx\\
	&=\Gamma(\nu,2x) e^x x^s |_{0}^{\infty}  - \int_{0}^{\infty} ((\Gamma(\nu,2x))'+\Gamma(\nu,2x))e^{x}x^{s}dx\\
	%&=\Gamma(\nu,2x) e^x \frac{x^s}{s} |_{0}^{\infty}  - \int_{0}^{\infty} (\Gamma'(\nu,2x)+\Gamma(\nu,2x)+)e^{x}\frac{x^{s}}{s} dx\\
	%&=- \int_{0}^{\infty} (-2(2x)^{\nu-1}e^{-2x}+\Gamma(\nu,2x))e^{x}\frac{x^{s}}{s} dx\\
	&=2^{\nu} \int_{0}^{\infty} x^{s+\nu-1}e^{-x}dx-\int_{0}^{\infty}\Gamma(\nu,2x)e^{x}x^{s} dx.
	\end{align*}	
The result follows from the definitions of the Gamma function. 
\end{proof}

%%%%%%%%%%%%%%%%%%%%%%%%%%%%%%%%%%
\section{Polynomial growth harmonic Maass forms}\label{PolyGrowthHMFs}
\subsection{Notation}
Throughout, $k \in \frac{1}{2} \mathbb{Z}$, $\tau = x+iy \in \mathbb{H}$, with $x,y  \in \mathbb{R}$, $y>0$ and $N \in \mathbb N.$ Further $q:=e^{2 \pi i \tau}.$

We recall that $\SL_2(\mathbb{Z})$ acts on $\mathbb{H}$ by 
$$
\left(\begin{matrix} a & b \\ c & d \end{matrix} \right ) \cdot z := \frac{az+b}{cz+d}.
$$
We let $\Gamma_0(N)$ be the subgroup of $\SL_2(\mathbb{Z})$ defined by
$$
\Gamma_0(N) := \left\{ \left ( \begin{matrix} a & b \\ c & d \end{matrix} \right ) \in \SL_2(\mathbb{Z}): N|c \right\}.
$$

We choose the principal branch of the square-root throughout. The {\it (Petersson) slash operator} is defined as 
\begin{align*}
\left(f\vert_k\gamma\right)(\tau) := \begin{cases}
(c\tau+d)^{-k} f(\gamma\tau) & \text{if } k \in \mathbb{Z}, \\
\left(\frac{c}{d}\right)\varepsilon_d^{2k}(c\tau+d)^{-k} f(\gamma\tau) & \text{if } k \in \frac{1}{2}+\mathbb{Z},
\end{cases}
\quad \gamma = \left(\begin{matrix} a & b \\ c& d \end{matrix}\right) \in \begin{cases}
\operatorname{SL}_2(\mathbb{Z}) & \text{if } k \in \mathbb{Z}, \\
\Gamma_0(4) & k \in \frac{1}{2}+\mathbb{Z},
\end{cases}
\end{align*}
where $\left(\frac{c}{d}\right)$ denotes the Kronecker symbol, and
\begin{align*}
\varepsilon_d := \begin{cases}
1 & \text{if} \ d \equiv 1 \pmod{4}, \\
i & \text{if} \ d \equiv 3 \pmod{4},
\end{cases}
\end{align*}
($d$ is guaranteed to be odd whenever $\gamma \in \Gamma_0(4)$).  The {\it weight $k$ hyperbolic Laplace operator} is given by
\begin{align*}
\Delta_k := -y^2\left(\frac{\partial^2}{\partial x^2}+\frac{\partial^2}{\partial y^2}\right) + iky\left(\frac{\partial}{\partial x} + i\frac{\partial}{\partial y}\right),
\end{align*}
and this operator decomposes as 
\begin{align} \label{eq:Deltasplitting}
\Delta_{k} = - \xi_{2-k} \xi_{k},
\end{align}
where $\xi_k$ acts on functions $f(\tau)$ by $$\xi_k(f)(\tau)=2iy^k\overline{\frac{\partial f}{\partial \overline\tau}}.$$
In particular, holomorphic functions are automatically annihilated by $\Delta_k$.
Further, this operator intertwines with the slash operator in ``dual weights'' $k$ and $2-k$ in the sense that for all $f\colon \mathbb H\rightarrow\mathbb C$ and all $\gamma\in\Gamma,$
$$
\xi_k(f|_k\gamma)=\left(\xi_k(f)\right)|_{2-k}\gamma
.
$$

\subsection{Harmonic Maass forms of polynomial growth}
Suppose that $4|N$ whenever $k \in \frac{1}{2} + \mathbb{Z}$. Let $f\colon \mathbb{H} \to \mathbb{C}$ be a smooth function. 
\begin{defn} Let $\chi$ be a character of $\Gamma_0(N).$ We call $f$ a {\it harmonic Maass form of polynomial growth of weight $k$ and character $\chi$} on $\Gamma_0(N)$, if $f$ satisfies the following conditions:
\begin{enumerate}
\item For all $\gamma \in \Gamma_0(N)$, we have $f|_k \gamma = \chi(\gamma)f$.
\item We have $\Delta_k f = 0$.
\item The function $f$ has polynomial growth at all the cusps of $\Gamma_0(N)$.
\end{enumerate} \label{def:hmfpg}
\end{defn}
We let $H_k^{\text{Eis}}(N,\chi)$ denote the space of such functions. Sometimes we refer to harmonic Maass forms of polynomial growth as harmonic Maass--Eisenstein series, which is the basis for our notation $H_k^{\text{Eis}}$.

We retrieve the space of weight $k$ holomorphic modular forms for level $N$ as the spaces
$$M_k(N,\chi):=\{f \in H_k^{\text{Eis}}(N,\chi); \text{$f$ is holomorphic in $\mathbb H$}\}.$$ 

\begin{lemma}\label{lem:HM_with_moderate_classicalMF}
    If $k > 2$, then $H_k^{\text{Eis}}(N,\chi) = M_k(N,\chi)$. 
\end{lemma}
\begin{proof}
If $f \in H_k^{Eis}(N,\chi)$, then by the above intertwining property of $\xi_k$, the decomposition of $\Delta_k$ into a composition of two $\xi$-operators, and the fact that the kernel of $\xi_k$ is the set of holomorphic functions, we have that $\xi_k(f) \in M_{2-k}(N,\chi)$. Since $k > 2$, $M_{2-k}(N,\chi)=\{0\}$ and thus $f$ is holomorphic.
\end{proof}

The polynomial growth condition and the equation $\Delta_k f = 0$ immediately imply the following canonical Fourier expansion of a harmonic Maass form of polynomial growth.
\begin{lemma}
    If $f \in H_k^{\text{Eis}}(N,\chi)$, then $f(\tau)$ has a Fourier expansion of the form
    \begin{equation}\label{eq:FourierExpansion}
    f(\tau) = \sum_{n=0}^{\infty} c^+(n) q^n + c^{-}(0) y^{1-k} + \sum_{n=1}^{\infty} c^-(n)\Gamma(1-k, 4 \pi n y) q^{-n}. 
    \end{equation}
\end{lemma}

The functional equation for the $L$-series attached to modular forms involves the Fricke involution $w_N$. We define the action of $w_N$ on $f \in H_k^{\text{Eis}}(N,\chi)$ as
\begin{equation}\label{eq:fricke}
  f|w_N (\tau) :=  N^{k/2} (N \tau)^{-k} f(-1/N\tau)  
\end{equation}

\begin{proposition}\label{f|}
If $f \in H_k^{\text{Eis}}(N,\chi)$, then $$f|w_N \in 
\begin{cases}
   H_k^{\text{Eis}}(N, \overline{\chi} ) & k \in \mathbb{Z} \\
   H_k^{\text{Eis}}(N,\overline{\chi} \left ( \frac{N}{\bullet} \right ) ) & k \in \frac{1}{2} + \mathbb{Z}.
\end{cases}$$
\end{proposition}
\begin{proof}
This was stated in \cite{ShanSingh22} for integral $k$. For the half-integral $k$, the transformation law is a consequence of Proposition 1.4 of \cite{Shimura1973}. Further,
working as in the proof of Lemma 5.2 of \cite{thebook}, we deduce
\begin{align}\label{eq:intertwining-property}
     \xi_k (f|_k w_N) 
     %&= \xi_k(N^{-k/2} \tau^{-k} f(-1/N\tau)) \nonumber \\
     %&= N^{-k/2} i y^k \left( \overline{\frac{\partial}{\partial x}  \tau^{-k} f(-1/N\tau))} - i \overline{\frac{\partial}{\partial y} (\tau^{-k} f(-1/N\tau))} \right) \nonumber \\
    %&= N^{-k/2} i y^k \left( \overline{\tau^{-k} \frac{\partial}{\partial x} ( f(-1/N\tau))} - i \overline{ \tau^{-k} \frac{\partial}{\partial y} ( f(-1/N\tau))} \right) \nonumber \\
    %&= N^{-k/2} i y^k \overline{\tau^{-k}} \left( \overline{\frac{\partial}{\partial x} ( f(-1/N\tau))} - i \overline{ \frac{\partial}{\partial y} ( f(-1/N\tau))} \right) \nonumber \\
   % &= N^{-k/2} i y^k \overline{\tau^{-k}} (-N \tau)^{-2} \left( \overline{\frac{\partial}{\partial x} ( f)}(-1/N\tau) - i \overline{ \frac{\partial}{\partial y} ( f)}(-1/N\tau) \right) \nonumber \\
    %&= N^{-k/2} i \operatorname{Im (-1/N \tau)}^k N^{2k}  \tau^{k} (- N \tau)^{-2} \left( \overline{\frac{\partial}{\partial x} ( f)}(-1/N\tau) - i \overline{ \frac{\partial}{\partial y} ( f)}(-1/N\tau) \right) \nonumber \\
     %&= N^{(k-2)/2} N i \operatorname{Im (-1/N \tau)}^k  i (-N \tau)^{k-2} \left( \overline{\frac{\partial}{\partial x} ( f)}(-1/N\tau) - i \overline{ \frac{\partial}{\partial y} ( f)}(-1/N\tau) \right) \nonumber \\&
    = N i \xi_k (f) |_{2-k} w_N,
\end{align}
which verifies that $f|w_N$ is harmonic. 
\end{proof} 

We will also need the growth of coefficients of polynomial growth harmonic Maass forms. For negative integer $k$, we quote the following from \cite{ShanSingh22}.
\begin{lemma}[Proposition 3.1 (iv) of \cite{ShanSingh22}]
If $k \in \mathbb{Z}_{<0}$ and $f \in H_k^{\text{Eis}}(N, \chi)$ %{\bf don't we work throughout only we $H^{\text{Eis}}_k$? The Eis is missing in some of the notation.} 
with Fourier expansion as in \eqref{eq:FourierExpansion}, then $c^{\pm}(n) = O(1)$ as $n \to \infty.$   
\end{lemma}

We will also need to consider $k=0, 2$ and half-integer values of $k < 2$.
\begin{proposition}\label{thm:MockCoeffsBound}
  Let $k=0, \, 2$ or $k \in \frac{1}{2}+\mathbb Z$ with $k<2$. If $f \in H_k^{\text{Eis}}(N,\chi)$ with Fourier expansion as in \eqref{eq:FourierExpansion}, then $c^+(n) = O(n^{\mu^{+}})$ and $c^-(n) = O(n^{\mu^-})$ as $n \to \infty$ for some $\mu^{\pm} \in \mathbb{R}$.
\end{proposition}
\begin{proof}
 The operator $\xi_k$ maps any $f \in H_k^{\text{Eis}}(N, \chi)$ to a modular form of weight $2-k$. On the other hand, the $n$-th Fourier coefficient of $f$ is (a constant multiple of) $\overline{c^-(n)}n^{1-k}$. Since the Fourier coefficients of cusp forms have polynomial growth, so will $c^-(n).$ 
 %\textcolor{blue}{RG: Since $c^{-}(n)n^{1-k}$ is the cusp form of weight $2-k$, then $c^{-}(n)n^{1-k}\leq n^{1-k/2}\Rightarrow c^{-}(n)\leq n^{k/2}$. Isn't then we have $\mu^{-}=k/2$? ND: Yes, but do we need to be so precise? We only want to show polynomial growth. RG: I was just checking, if we don't want to be this precise, we may remove the comment and I am ok with it.} 
 For $c^+(n)$, we note that 
\begin{equation}\label{coeff}
c^+(n)e^{-2 \pi n y}=\int_0^1f(x+iy)e^{-2 \pi i n x}dx.
\end{equation}
Since $f$ has polynomial growth for all cusps, we deduce that there are $A, C>0$, such that $|f(\tau)| \le Cy^{A}$, for all $\tau \in \mathbb H.$ Therefore,
$$|c^+(n)|e^{-2 \pi n y} \le C y^A$$ and, for $y=1/n$, 
$c^+(n) \ll n^B$ for some $B.$
\end{proof}

\subsection{$L$-functions for polynomial growth harmonic Maass forms}
We set
$$f(\tau) = \sum_{n=0}^{\infty} a^+(n) q^n + a^{-}(0) y^{1-k}+ \sum_{n=1}^{\infty} a^-(n)\Gamma(1-k, 4 \pi n y) q^{-n} \in H_k^{Eis} (N, \chi),$$
and
\begin{align*}
g(\tau) &=f|w_N(\tau)= N^{k/2} (N \tau)^{-k} f(-1/N\tau) \\
&= \sum_{n=0}^{\infty} b^+(n) q^n + b^-(0) y^{1-k} + \sum_{n=1}^{\infty} b^-(n) \Gamma(1-k, 4 \pi n y) q^{-n}
\end{align*}
which, by Proposition \ref{f|} belongs to $H_k^{Eis} (N, \bar \chi),$ if $k \in \mathbb Z$ or to $H_k^{Eis} (N, \bar{\chi}  
\left ( \frac{N}{\bullet}\right )),$ if $k \in \frac{1}{2}+\mathbb Z.$  
For $Re(s) > \max \{ \mu^+_f, \mu^-_f\}+1,$
\begin{equation}\label{eq:completedLfcn}
\Lambda(f,s) := \left( \frac{\sqrt{N}}{2 \pi } \right)^s  \left( \Gamma(s) L^+(f,s) + W_{1-k}(s) L^-(f,s) \right),
\end{equation}
where
$W_{\nu} (s)$ is defined in the previous section and 
$L^{\pm}(f,s) := \sum_{n=1}^{\infty} \frac{a^{\pm} (n)}{n^s}.$

We can now establish the functional equation for $\Lambda(f, s)$.
\begin{theorem}\label{thm:ShankadharSingh}
Let $k \in \frac{1}{2}\mathbb{Z}$ 
%be nonpositive or in $\{ \frac{1}{2}, \frac{3}{2} \}$
. The function $\Lambda(f,s)$ has a meromorphic continuation to $\mathbb{C}$ with poles at  
$s =0,1,k-1,k$, with residues $-a^+(0), i^k b^-(0)N^{\frac{k-1}{2}}, -a^-(0) N^{\frac{k-1}{2}},i^k b^+(0)$.  Moreover, we have the functional equation
\begin{equation}\label{eq:fcnleq}
\Lambda ( f, s) = i^k \Lambda(g, k-s).
\end{equation}
\end{theorem}
\begin{remark}
The analogue for negative integers $k$ is Theorem 1 of \cite{ShanSingh22}.
\end{remark}
\begin{proof}
While this result is stated for negative integral weights, the method in \cite{ShanSingh22} of splitting up the integral and making $u$-substitutions  works for half-integral $k < 1$ as well. 

Let $f^*(\tau):= f(\tau) - a^+(0) - a^-(0)y^{1-k}$ and $g^*(\tau):= g(\tau) - b^+(0) - b^-(0)y^{1-k}$. Then for $Re(s) > \max \{\mu_f^{\pm} +1,k-1\}$, we find that
\begin{align*}
\int_0^{\infty} &\left( f(it/\sqrt{N}) - a^+(0) - \frac{a^-(0)}{N^{(1-k)/2}} t^{1-k} \right) t^{s-1} dt \\
&= \sum_{n=1}^{\infty} a^+ (n) \int_0^{\infty} e^{-2 \pi nt/\sqrt{N}} t^{s-1} dt  + \sum_{n=1}^{\infty} a^- (n) \int_0^{\infty} e^{2 \pi nt/\sqrt{N}} \Gamma(1-k, 4 \pi nt/\sqrt{N}) t^{s-1} dt \\
&= \left( \frac{\sqrt{N}}{2 \pi } \right)^s  \left( \Gamma(s) L^+(f,s) + W_{1-k}(s) L^-(f,s) \right)= \Lambda(f,s).
\end{align*}
Since $f(i/n\sqrt{N})=f(-1/(N(ui/\sqrt{N}))=g(iu/\sqrt{N})(iu)^k,$ the change of variables $u=1/t$ gives
%So for $Re(s) > \max \{\mu^{\pm},k-1\}$, we have
\begin{align} \label{eq:gamma_as_integrals}
\Lambda(f,s) 
%&= \int_0^{\infty} f^*(it/\sqrt{N}) t^{s-1} dt \\ \nonumber
&= \int_0^{1} f^*(it/\sqrt{N}) t^{s-1} dt  + \int_1^{\infty} f^*(it/\sqrt{N}) t^{s-1} dt \\ \nonumber
%&= \int_1^{\infty} f^*(i/u\sqrt{N}) u^{-s-1} du  + \int_1^{\infty} f^*(it/\sqrt{N}) t^{s-1} dt \\\nonumber
&= \int_1^{\infty} f(i/u\sqrt{N}) u^{-s-1} du + \int_1^{\infty} f^*(it/\sqrt{N}) t^{s-1} dt -\int_1^{\infty} (a^+(0) + a^-(0)(u \sqrt{N})^{k-1}) u^{-s-1} du  \\\nonumber
&= i^k \int_1^{\infty} g(iu/\sqrt{N}) u^{k-s-1} du + \int_1^{\infty} f^*(it/\sqrt{N}) t^{s-1} dt - \frac{a^+(0)}{s} - \frac{a^-(0)N^{\frac{k-1}{2}}}{s-k+1}\\\nonumber
&=  i^k \int_1^{\infty} g^*(iu/\sqrt{N}) u^{k-s-1} du +  \frac{ i^kb^+(0)}{s-k} + \frac{ i^kb^-(0)N^{\frac{k-1}{2}}}{ s-1}  - \frac{a^+(0)}{s} - \frac{a^-(0) N^{\frac{k-1}{2}}}{s-k+1} \\ \nonumber
& \hspace{2cm}+ \int_1^{\infty} f^*(it/\sqrt{N}) t^{s-1} dt   
\end{align}
Because of the exponential decay as $t \to \infty$ of $f^*(it)$ and $g^*(it)$, the two integrals are entire, and we see that $\Lambda(f,s)$ has simple poles at $s =0,1,k-1,k$, with residues $-a^+(0),  i^k b^-(0)N^{\frac{k-1}{2}}$, $-a^-(0) N^{\frac{k-1}{2}}$,  $i^k b^+(0)$. 
On the other hand, $g|w_N=(f|w_N)|w_N=i^{-2k}f$ and hence and application of \eqref{eq:gamma_as_integrals} to $g=f|w_N$ instead of $f$ gives, for $\Re(s) > \max \{\mu^{\pm}_g,k-1\}$,
\begin{align*}
    \Lambda(g,s)=
    %&= \int_1^{\infty} g(i/u\sqrt{N}) u^{-s-1} du + \int_1^{\infty} g^*(it/\sqrt{N}) t^{s-1} dt - \frac{b^+(0)}{s} - \frac{b^-(0)N^{\frac{k-1}{2}}}{s-k+1}\\
%    \tag{Using the transformation relation between $f$ and $g$,}
   % &= \int_1^{\infty} N^{k/2}(\sqrt{N}i/u)^{-k}f(-u\sqrt{N}/iN) u^{-s-1} du + \int_1^{\infty} g^*(it/\sqrt{N}) t^{s-1} dt  \\
    %& \hspace{1cm} - \frac{b^+(0)}{s} - \frac{b^-(0)N^{\frac{k-1}{2}}}{s-k+1}\\
    %&= \int_1^{\infty} i^{-k}f(ui/\sqrt{N}) u^{k-s-1} du + \int_1^{\infty} g^*(it/\sqrt{N}) t^{s-1} dt 
    %- \frac{b^+(0)}{s} - \frac{b^-(0)N^{\frac{k-1}{2}}}{s-k+1}\\
    %&= \int_1^{\infty} i^{-k}(f^*(ui/\sqrt{N})+a^+(0)+a^-(0)(u/\sqrt{N})^{1-k}) u^{k-s-1} du + \int_1^{\infty} g^*(it/\sqrt{N}) t^{s-1} dt \\
    %& \hspace{1cm} - \frac{b^+(0)}{s} - \frac{b^-(0)N^{\frac{k-1}{2}}}{s-k+1}\\
    %&= \int_1^{\infty} i^{-k}f^*(ui/\sqrt{N} u^{k-s-1} du + i^{-k} \int_1^{\infty} a^+(0) u^{k-s-1} du + i^{-k} \int_1^{\infty} a^-(0)N^{\frac{k-1}{2}} u^{-s} du \\
    %& \hspace{1cm}+ \int_1^{\infty} g^*(it/\sqrt{N}) t^{s-1} dt - \frac{b^+(0)}{s} - \frac{b^-(0)N^{\frac{k-1}{2}}}{s-k+1}\\
&
i^{-k} \int_1^{\infty} f^*(iu/\sqrt{N}) u^{k-s-1} du - \frac{ i^{-k}a^+(0)}{k-s} + \frac{ i^{-k}a^-(0)N^{\frac{k-1}{2}}}{ s-1}  - \frac{b^+(0)}{s} - \frac{b^-(0) N^{\frac{k-1}{2}}}{s-k+1} \\ 
   & \hspace{1cm} + \int_{1}^{\infty} g^*(it/\sqrt{N})t^{s-1} dt.
\end{align*} 
Replacing $s$ with $k-s$ and comparing to \eqref{eq:gamma_as_integrals}, we deduce the result.
%\begin{align*}\Lambda(g,k-s) &= i^{-k} \int_1^{\infty} f^*(iu/\sqrt{N}) u^{s-1} du - \frac{ i^{-k}a^+(0)}{s} + \frac{ i^{-k}a^-(0)N^{\frac{k-1}{2}}}{k- s-1}  - \frac{b^+(0)}{k-s} - \frac{b^-(0) N^{\frac{k-1}{2}}}{1-s} \\ 
  % & \hspace{1cm} + \int_{1}^{\infty} g^*(it/\sqrt{N})t^{k-s-1} dt\\
  % &= i^{-k} \Lambda(f,s)    
%\end{align*}
\end{proof}

The next two bounds will be crucial for establishing the summation formula in Theorem \ref{thm:summationformula}.
\begin{lemma}\label{thm:LambdaStirlingBound}
Let $f \in H_k(N,\chi)$, and let $\alpha > 1 + \operatorname{max} \{\mu_f^{\pm} \}$. As $|t| \to \infty$,
$$
\Lambda(f, \alpha+it)=O_{\alpha,f} (|t|^{\alpha- \frac{1}{2} + \operatorname{max} \{ 1-k, 0 \}} e^{- \pi |t| /2}).
$$
\end{lemma}
\begin{proof} Since $\alpha > 1 + \operatorname{max} \{\mu_f^{\pm} \}$, 
$|L^{\pm} (f, \alpha + it)| \le \sum_{n = 1}^{\infty} |c^{\pm}(n)|/n^{\alpha} =O(1).$ With Theorem \ref{thm:st} and Corollary \ref{thm:Wbound}, we have as $|t| \to \infty$: 
\begin{align*}
\Lambda(f, \alpha+it) &=  \left( \frac{\sqrt{N}}{2 \pi } \right)^{\alpha + it}  \left( \Gamma(\alpha + it) L^+(f, \alpha + it) + W_{1-k}(\alpha + it) L^-(f,\alpha+it) \right)  \\
%\\&\le | \left( \frac{\sqrt{N}}{2 \pi } \right)^{\alpha}  \Gamma(\alpha + it) L^+(f, \alpha + it)| + | \left( \frac{\sqrt{N}}{2 \pi } \right)^{\alpha} W_{1-k}(\alpha + it) L^-(f,\alpha+it)  | \\
& \ll_{f, \alpha} |t|^{\alpha-1/2}e^{-\pi|t|/2} + |t|^{\alpha +1-k- \frac{1}{2}} e^{-\pi |t| /2} \\
%&\le  \left( \frac{\sqrt{N}}{2 \pi } \right)^{\alpha} (1 + \delta) |t|^{\alpha-1/2}e^{-\pi|t|/2}C^+ + \left( \frac{\sqrt{N}}{2 \pi } \right)^{\alpha} (1 + \delta) C_{\alpha} |t|^{\alpha +1-k- \frac{1}{2}} e^{-\pi |t| /2} C^- \\
%&=  \left( \frac{\sqrt{N}}{2 \pi } \right)^{\alpha} \left( (1 + \delta) C^+ + (1 + \delta) C_{\alpha} |t|^{ 1-k }  C^- \right) e^{-\pi |t| /2} |t|^{\alpha-1/2} \\
%&=  \left( \frac{\sqrt{N}}{2 \pi } \right)^{\alpha} \left( (1 + \delta) C^+|t|^{\alpha-1/2} + (1 + \delta) C_{\alpha} |t|^{ 1-k + \alpha-1/2 }  C^- \right) e^{-\pi |t| /2}  \\
%&=A_{f,\alpha} |t|^{\max(\alpha-1/2,\alpha+1/2-k)}e^{-\pi |t| /2} 
\\
& \ll_{f, \alpha}|t|^{\alpha - 1/2 +\max\{0,1-k\}}e^{-\pi |t| /2}.
\end{align*}
\end{proof}
\begin{proposition}\label{thm:gamma_bound_for_large_t}
Suppose $f \in H_k(N,\chi)$ and let $g=f|w_N$. Let $\alpha >  \operatorname{max} \{1+\mu^{\pm}_{f},1+\mu^{\pm}_{g}, \frac{k}{2}\}$, $\rho>0$. For $\sigma \in [k-\alpha, \alpha]$, 
we have
$$
\left| \frac{\Lambda(f, \sigma+it) }{\Gamma(\rho + 1 + \sigma + it)} \right| = O(|t|^{-\sigma-1+\alpha-\rho+\max (1-k, 0)})
$$
as $|t|\to \infty$, uniformly for $\sigma$ in $[k-\alpha,\alpha]$.
\end{proposition}
\begin{proof} 
For $\sigma = \alpha,$ we use Lemma \ref{thm:LambdaStirlingBound} and Proposition~\ref{thm:st} to obtain as $|t| \to \infty$
\begin{align}\label{stirlingq}
\frac{\Lambda(f, \alpha+it) }{\Gamma(\rho + 1 + \alpha + it)} \ll \frac{ e^{- \pi |t| / 2} |t|^{\alpha - \frac{1}{2} + \operatorname{max}\{1-k,0 \}}}{|t|^{\rho +1 + \alpha- \frac{1}{2}} e^{- \pi |t|/2}} = O(|t|^{-\rho -1+\max (1-k, 0)}). 
\end{align}
When $\sigma = k-\alpha$, with Theorem \ref{thm:ShankadharSingh} and and application of \eqref{stirlingq} combined with Theorem \ref{thm:st}, we obtain, as $|t| \to \infty$,
$$\frac{\Lambda(f, k-\alpha+it) }{\Gamma(\rho + 1 + k-\alpha + it)}=
i^k\frac{\Lambda(g, \alpha-it) }{\Gamma(\rho + 1 + \alpha - it)} \cdot
\frac{\Gamma(\rho + 1 + \alpha - it)}{\Gamma(\rho + 1 + k-\alpha + it)} \ll |t|^{-\rho-1+\max(1-k, 0)+2 \alpha-k}.$$
Therefore, with Proposition~\ref{thm:pl}, we deduce the claim.
\end{proof}

%%%%%%%%%%%%%
%%%%%%%%%%%%%%%%%%%%%%%%%%
%%%%%%%%%%%%%%%%%%%%%%%%%%
%%%%%%%%%%%%%%%%%%%%%%%%%%
%%%%%%%%%%%%%
\section{A summation formula for polynomial growth harmonic Maass forms}\label{SummationSection}
With the preparatory work of the previous sections, we are now ready to state and prove our summation formula. 
\begin{theorem}\label{thm:summationformula}
Let
\begin{equation}
    \label{FEf}
f(\tau) = \sum_{n=0}^{\infty} a^+(n) q^n + a^-(0) y^{1-k} + \sum_{n=1}^{\infty} a^-(n) \Gamma(1-k, 4 \pi n y) q^{-n} \in H_k^{\text{Eis}} (N, \chi),
\end{equation}
and
$$
g(\tau) = f | w_N (\tau) = \sum_{n=0}^{\infty} b^+(n) q^n + b^-(0) y^{1-k} + \sum_{n=1}^{\infty} b^-(n) \Gamma(1-k, 4 \pi n y) q^{-n},
$$
and assume that $\mu_f^{\pm}, \mu_g^{\pm} \in \mathbb{R}^+$ are such that $a^{\pm}(n) = O(n^{\mu_f^{\pm}})$ and $b^{\pm}(n) = O(n^{\mu_g^{\pm}})$, respectively. 
We choose $\rho > \rho_0 - \frac{1}{2}$, where
$$
\rho_0 := \begin{cases}
\operatorname{max} \{ 2+ 2\mu_f^{\pm} -k, 2 + 2\mu_g^{\pm} - k, k \} & k \ge 1, \\
\operatorname{max} \{ 3 + 2\mu_f^{\pm} -2k, 3 + 2\mu_g^{\pm} - 2k \} & k <1.    
\end{cases}
$$

Then we have
\begin{align} \label{eq:sumformula}
&\frac{1}{\Gamma(\rho+1)}\sum_{n\leq x}a^{+}(n)(x-n)^{\rho} +\frac{x^{\rho}}{2 \pi i } \sum_{n \le x} a^-(n) g_{\rho}\left(n, x \right) - Q_{\rho}(x) \nonumber\\
&=-i^k x^{(\rho + k)/2}\left(\frac{\sqrt{N}}{2\pi }\right)^{\rho}\sum_{n=1}^{\infty}\frac{b^{+}(n)}{n^{(\rho+k)/2}}J_{\rho+k}\left(4\pi\sqrt{\frac{nx}{N}}\right)\nonumber\\
&- i^k x^{(\rho+1)/2}\left(\frac{\sqrt{N}}{2\pi}\right)^{\rho+k-1}\sum_{n=1}^{\infty} \frac{b^-(n)}{n^{(\rho-1)/2+k}}\int_{0}^{1/2}\frac{u^{k+(\rho-3)/2}}{(1-u)^{(1+\rho)/2}
}J_{\rho+1}\left(4\pi\sqrt{\frac{nx(1-u)}{Nu}}\right)du
\end{align}
where
\begin{equation}
    \label{Qrho}
Q_{\rho}(x) = x^{\rho} \left( -\frac{a^+(0)}{\Gamma(\rho + 1)} + \frac{2 \pi x N^{\frac{k}{2}-1} b^-(0) i^k}{\Gamma(\rho + 2)} - \frac{ a^-(0) x^{k-1} ( 2 \pi )^{k-1}}{ \Gamma(\rho +  k)}  + \frac{b^{+}(0) i^k x^{k} (2 \pi)^k}{N^{\frac{k}{2}} \Gamma(\rho + k + 1)} \right)
\end{equation}
  and 
$g_{\rho}(n,x)$ as given in Theorem~\eqref{thm:perrongen}.
\end{theorem}
\begin{remark}\begin{enumerate}
    \item Given Lemma \ref{lem:HM_with_moderate_classicalMF}, it is natural to assume $k \le 2$ because for $k>2$, the harmonic Maass form must be a holomorphic modular form. In this case one can apply a classical summation formula such as Voronoi summation \cite{copston}.
    \item When $f$ is holomorphic modular form, and thus $a^-(n) = 0 = b^-(n)$, we retrieve the formula given by Lemma 5 of \cite{ChandNaras61}. \end{enumerate}\end{remark}
\begin{proof}
We will first prove \eqref{eq:sumformula} for $\rho>\rho_0$. 
Let $\alpha \in \mathbb{R}$ satisfy
\begin{equation}\label{range} \operatorname{max} \{1+ \mu_f^{\pm},1 + \mu_g^{\pm}, k, k/2\} < \alpha < \operatorname{min} \Bigg\{\frac{k+\rho}{2}, \frac{\rho + 2k-1}{2} \Bigg\}.
\end{equation}
Note that our lower bound on $\rho$ makes such a choice possible. Then Proposition~\ref{thm:perron} gives % generalized Perron's forumla
\begin{align}\label{1}
\frac{1}{\Gamma(\rho+1)}\sum_{n\leq x} &a^{+}(n)(x-n)^{\rho}=\frac{1}{2\pi i}\int_{(\alpha)}\frac{\Gamma(s)L^{+}(f,s)}{\Gamma(\rho+1+s)}x^{s+\rho}\ ds\nonumber\\
&=\frac{x^\rho}{2\pi i}\int_{(\alpha)}\frac{\Lambda(f,s)}{\Gamma(\rho+1+s)}\left(\frac{2\pi x}{\sqrt{N}} \right)^{s} ds  -\frac{x^\rho}{2\pi i}\int_{(\alpha)}\frac{W_{1-k}(s)L^{-}(f,s)}{\Gamma(\rho+1+s)} x^s  ds.\nonumber
\end{align}
%since $\rho\geq 0$
% Now we analyze each integral separately.

We will first evaluate the first integral on the right-hand side in \eqref{1} by shifting the line of integration to $k-\alpha<0$. By Theorem \ref{thm:ShankadharSingh}, $\Lambda(f,s)$ has poles at $s=0,1,k-1$, and $k$, by Cauchy's Residue Theorem,  we find
\begin{align}
\frac{x^{\rho}}{2\pi i}\int_{(\alpha)}\frac{\Lambda(f,s)}{\Gamma(\rho+1+s)}\left(\frac{2\pi x}{\sqrt{N}} \right)^{s}\ ds= \sum_{j \in \{0,1,k-1,k \}} R_j +\frac{x^{\rho}}{2\pi i}\int_{(k-\alpha)}\frac{\Lambda(f,s)}{\Gamma(\rho+1+s)}\left(\frac{2\pi x}{\sqrt{N}} \right)^{s}\ ds,
\end{align}
where $R_j$ is the residue of $\frac{x^{\rho} \Lambda(f,s)}{\Gamma ( \rho + 1 + s)} \left( \frac{2 \pi x}{\sqrt{N}} \right)^s$ at $s=j$. Note that, by Proposition \ref{thm:gamma_bound_for_large_t}, the contribution from the upper and lower horizontal line integrals tends to zero.
Set $ Q_{\rho} (x) =R_0+R_1+R_k+R_{k-1}$, referred to as the residual function. Using Theorem \ref{thm:ShankadharSingh}, we have
\begin{align*}
Q_{\rho} (x) &= -\frac{a^+(0) x^{\rho}}{\Gamma(\rho + 1)} + \frac{2 \pi x^{1 + \rho} N^{\frac{k-1}{2}} b^-(0) i^k}{N^{\frac{1}{2}}\Gamma(\rho + 2)} - \frac{ a^-(0) N^{\frac{k-1}{2}} x^{k-1+\rho} ( 2 \pi )^{k-1}}{N^{\frac{k-1}{2}} \Gamma(\rho +  k)}  + \frac{b^{+}(0) i^k x^{k + \rho} (2 \pi)^k}{N^{\frac{k}{2}} \Gamma(\rho + k + 1)} \\
&= x^{\rho} \left( -\frac{a^+(0)}{\Gamma(\rho + 1)} + \frac{2 \pi x N^{\frac{k}{2}-1} b^-(0) i^k}{\Gamma(\rho + 2)} - \frac{ a^-(0) x^{k-1} ( 2 \pi )^{k-1}}{ \Gamma(\rho +  k)}  + \frac{b^{+}(0) i^k x^{k} (2 \pi)^k}{N^{\frac{k}{2}} \Gamma(\rho + k + 1)} \right).
\end{align*}
Using the functional equation \eqref{eq:fcnleq} and taking into account the range of $\alpha$, we obtain
\begin{multline}\label{eq:4.3}
\frac{x^\rho}{2\pi i}\int_{(\alpha)}\frac{\Lambda(f,s)}{\Gamma(\rho+1+s)}\left(\frac{2\pi x}{\sqrt{N}} \right)^{s}\ ds = Q_{\rho}(x) +\frac{x^\rho}{2\pi i} i^k \int_{(k-\alpha)}\frac{\Lambda(g,k-s)}{\Gamma(\rho+1+s)}\left(\frac{2\pi x}{\sqrt{N}} \right)^{s}\ ds \\
= Q_{\rho}(x) - \frac{i^k x^\rho}{2\pi i} \left(\frac{2\pi x}{\sqrt{N}} \right)^{k} \int_{(\alpha)}\frac{\Lambda(g,s)}{\Gamma(\rho+1+k-s)}\left(\frac{2\pi x}{\sqrt{N}} \right)^{-s}\ ds  \\
= Q_{\rho}(x)-i^k x^\rho \left(\frac{2\pi x}{\sqrt{N}} \right)^{k} \sum_{n=1}^{\infty} b^+(n) \frac{1}{2\pi i}\int_{(\alpha)}\frac{\Gamma(s)}{ \Gamma(\rho+1+k-s)}\left(\frac{4\pi^2 nx}{N} \right)^{-s}\ ds  \\ 
 - i^k x^\rho \left(\frac{2\pi x}{\sqrt{N}} \right)^{k} \sum_{n=1}^{\infty} b^-(n)  \frac{1}{2\pi i}\int_{(\alpha)}\frac{W_{1-k}(s) }{ \Gamma(\rho+1+k-s)}\left(\frac{4\pi^2 nx}{N} \right)^{-s}\ ds. 
\end{multline}
The interchange of summation and integration in the first (resp. second) integral is justified by Proposition \ref{thm:stirling} (resp. Corollary \ref{thm:Wbound}) together with the choice of $\alpha$ in \eqref{range}.
With this, our summation formula is 
\begin{align}\label{2}
\frac{1}{\Gamma(\rho+1)}\sum_{n\leq x} &a^{+}(n)(x-n)^{\rho} = Q_{\rho}(x)\\
&-i^k x^\rho \left(\frac{2\pi x}{\sqrt{N}} \right)^{k} \sum_{n=1}^{\infty} b^+(n) \frac{1}{2\pi i}\int_{(\alpha)}\frac{\Gamma(s)}{ \Gamma(\rho+1+k-s)}\left(\frac{4\pi^2 nx}{N} \right)^{-s}\ ds\nonumber  \\ 
& - i^k x^\rho \left(\frac{2\pi x}{\sqrt{N}} \right)^{k} \sum_{n=1}^{\infty} b^-(n)  \frac{1}{2\pi i}\int_{(\alpha)}\frac{W_{1-k}(s) }{ \Gamma(\rho+1+k-s)}\left(\frac{4\pi^2 nx}{N} \right)^{-s}\ ds\nonumber\\ 
&-\frac{x^\rho}{2\pi i}\int_{(\alpha)}\frac{W_{1-k}(s)L^{-}(f,s)}{\Gamma(\rho+1+s)} x^s  ds.\nonumber
\end{align} 
We will now show closed formulas for each contour integral in the right-hand side of \eqref{2}. 
For the first one we employ \cite[7.3(23)]{erd}:
  \begin{equation}\label{eq:besselj1}
  \frac{1}{2\pi i}\int_{(\alpha)} \frac{\Gamma(s)}{\Gamma(\rho+k-s+1)} z^{-s} ds = z^{-(\rho+k)/2}J_{\rho+k}(2\sqrt{z}).
  \end{equation}
which holds because $0 < \alpha \le \frac{\rho+k}{2}$. We obtain
\begin{equation} 
\label{5.5}
    \frac{1}{2\pi i}\sum_{n=1}^{\infty} b^+(n) \int_{(\alpha)}\frac{\Gamma(s)}{\Gamma(\rho+1+k-s)}\left(\frac{4\pi^2 nx}{N} \right)^{-s}\ ds=\left(\frac{\sqrt{N}}{2\pi \sqrt{x}}\right)^{k+\rho}\sum_{n=1}^{\infty}\frac{b^{+}(n)}{n^{\frac{\rho+k}{2}}}J_{\rho+k}\left(4\pi\sqrt{\frac{nx}{N}}\right). 
\end{equation}
For the second sum in the right-hand side of \eqref{2}, we employ Lemma \ref{thm:Wformula} to get
 \begin{multline}
\sum_{n=1}^{\infty} b^-(n) \frac{1}{2\pi i} \int_{(\alpha)}\frac{W_{1-k}(s) }{ \Gamma(\rho+1+k-s)}\left(\frac{4\pi^2 nx}{N} \right)^{-s}\ ds\nonumber\\
= \sum_{n=1}^{\infty} b^-(n) \frac{1}{2\pi i} \int_{(\alpha)}  \frac{\Gamma( s + 1 -k) B(\frac{1}{2}, s, k-s)}{\Gamma(\rho + 1 +k- s)}\left(\frac{4\pi^2 nx}{N} \right)^{-s}ds \nonumber\\
=\sum_{n=1}^{\infty} b^-(n) \int_{0}^{1/2}u^{-1}(1-u)^{k-1}\frac{1}{2\pi i}\int_{(\alpha)}  \frac{\Gamma( s + 1 -k)}{\Gamma(\rho + 1 +k- s)}\left(\frac{4\pi^2 x n(1-u)}{Nu} \right)^{-s}ds\ du \nonumber\\
=\sum_{n=1}^{\infty} b^-(n) \int_{0}^{1/2}u^{-1}(1-u)^{k-1}\frac{1}{2\pi i}\int_{(\alpha + 1 -k)}  \frac{\Gamma( v)}{\Gamma(\rho + 2 - v)}\left(\frac{4\pi^2 x n(1-u)}{Nu} \right)^{-(v + k -1)}dv\ du.  \nonumber
%\\
%&=\sum_{n=1}^{\infty} b^-(n) \int_{0}^{1/2}u^{-1}(1-u)^{k-1}\frac{1}{2\pi i} \left(\frac{4\pi^2 x n(1-u)}{Nu} \right)^{-(k -1)}  \int_{(\alpha + 1 -k)}  \frac{\Gamma( v)}{\Gamma(\rho + 2 - v)}\left(\frac{4\pi^2 x n(1-u)}{Nu} \right)^{-v }dv\ du  \nonumber\\
\end{multline}
The interchange of integrals at the second equality is justified by the absolute convergence of the inner integral guaranteed by $\alpha < \frac{\rho + 2k-1}{2}$. With \eqref{eq:besselj1}, this becomes
\begin{align}
&\sum_{n=1}^{\infty} b^-(n) \int_{0}^{1/2}u^{-1}(1-u)^{k-1} \left(\frac{4\pi^2 x n(1-u)}{Nu} \right)^{-(k -1) - \frac{\rho + 1}{2}}  J_{\rho + 1} \left (2 \sqrt{ \frac{4\pi^2 x n(1-u)}{Nu}} \right )du.
%\nonumber \\
%&=\left(\frac{\sqrt{N}}{2\pi\sqrt{x}}\right)^{\rho+2k-1}\sum_{n=1}^{\infty} \frac{b^-(n)}{n^{(\rho-1)/2+k}}\int_{0}^{1/2}\frac{u^{k+(\rho-3)/2}}{(1-u)^{(1+\rho)/2}
%}J_{\rho+1}\left(4\pi\sqrt{\frac{nx(1-u)}{Nu}}\right)du.
%\\ &= \int_{0}^{1/2} \frac{u^{k+(\rho-3)/2}}{(1-u)^{(1+\rho)/2}} \sum_{n=1}^{\infty} \left(\frac{\sqrt{N}}{2\pi\sqrt{x}}\right)^{\rho+2k-1}  \frac{b^-(n)}{n^{(\rho-1)/2+k}}J_{\rho+1}\left(4\pi\sqrt{\frac{nx(1-u)}{Nu}}\right)du. 
\label{5.6}
\end{align}
The identity for $\rho>\rho_0$ then follows by substituting into \eqref{2} according to \eqref{5.5}, \eqref{5.6} and the identity of Theorem \ref{thm:perrongen}. 

Having proved \eqref{eq:sumformula} for $\rho>\rho_0$, we can now extend its range of validity based on the observation that both series in its right-hand side converges absolutely for $\rho>\rho_0-\frac12.$ This is a direct consequence of the bounds 
\begin{align}\label{series}
 &  x^{\frac{\rho +k}{2}}\sum_{n=1}^{\infty}\frac{b^{+}(n)}{n^{(\rho+k)/2}}J_{\rho+k}\left(4\pi\sqrt{\frac{nx}{N}}\right) \ll
x^{\frac{\rho +k}{2}-\frac14}\sum_{n=1}^{\infty}\frac{1}{n^{\frac{\rho+k}{2}-\mu_g^+ +\frac14}} \nonumber \\
&\text{and} \nonumber
\\ \nonumber
&  x^{\frac{\rho +1}{2}}\sum_{n=1}^{\infty} \frac{b^-(n)}{n^{\frac{\rho-1}{2}+k}}\int_{0}^{\frac12}\frac{u^{k+\frac{\rho-3}{2}}}{(1-u)^{\frac{1+\rho}{2}}
}J_{\rho+1}\left(4\pi\sqrt{\frac{nx(1-u)}{Nu}}\right)du \\
&  \ll
x^{\frac{\rho +1}{2}-\frac14}\left (\sum_{n=1}^{\infty} \frac{1}{n^{\frac{\rho-1}{2}+k-\mu^-_g +\frac14}} \right ) \int_{0}^{\frac12}\frac{u^{k+\frac{\rho}{2}-\frac54}}{(1-u)^{\frac{\rho}{2}+\frac34}
}du,
\end{align}
which are derived from the bound $J_{\nu}(x) \ll_{\nu} x^{-1/2}$ implied by Proposition \ref{thm:besselj-asymptotics}. 

Let then $\rho>\rho_0-\frac12$ and $I$ an interval contained in some closed interval in $\mathbb R_{>0}$. 
For convenience we set $y=\sqrt{x}$ and rewrite the right-hand side of Theorem \ref{thm:summationformula} as a single series as follows:
\begin{equation}\label{eq:ArhoDef}
A_{\rho}(y) =: \sum_{n = 1}^{\infty} a_{\rho}(n,y),
\end{equation}
where
\begin{align*}\label{eq:Arhofunction}
a_{\rho}(n,y) &= -i^k y^{\rho + k}\left(\frac{\sqrt{N}}{2\pi }\right)^{\rho}\frac{b^{+}(n)}{n^{(\rho+k)/2}}J_{\rho+k}\left(4\pi y \sqrt{\frac{n}{N}}\right) \\
&- i^k y^{\rho+1}\left(\frac{\sqrt{N}}{2\pi}\right)^{\rho+k-1} \frac{b^-(n)}{n^{(\rho-1)/2+k}}\int_{0}^{1/2}\frac{u^{k+(\rho-3)/2}}{(1-u)^{(1+\rho)/2}
}J_{\rho+1}\left(4\pi y\sqrt{\frac{n(1-u)}{Nu}}\right)du.
\end{align*}
The absolute convergence of $A_{\rho}(y)$ for $\rho>\rho_0-\frac12$, established above with \eqref{series}, implies that $A_{\rho}(y)$ converges uniformly for $y \in I$. 

On the other hand, the identity $(z^{\nu} J_{\nu}(z))'=z^{\nu}J_{\nu-1}(z)$ ((10.6.6) of \cite{NIST:DLMF}) implies that
\begin{equation}
    \label{a'}
a'_{\rho+1}(n,y)=2y a_{\rho}(n,y).
\end{equation}
Therefore, since, as just shown, $2y \sum_{n \ge 1} a_{\rho}(n, y)=\sum_{n \ge 1} a'_{\rho+1}(n, y)$ converges uniformly for $y \in I$, when $\rho>\rho_0-\frac12$, we deduce that $\sum a'_{\rho+1}(n, y)=(\sum a'_{\rho+1}(n, y))$. Therefore,
\begin{equation}
    \label{deriv}
2yA_{\rho}(y)=2y \sum_{n \ge 1} a_{\rho}(n, y)=\sum_{n \ge 1} a'_{\rho+1}(n, y)=A'_{\rho+1}(y).
\end{equation}
Since $\rho+1>\rho_0$, the first part of the proof shows that \begin{equation}
    \label{deriv2}
    t_{\rho+1}(y^2)=A_{\rho+1}(y)\end{equation} where
\begin{equation}
    \label{trho} t_\rho(x):=\frac{1}{\Gamma(\rho+1)}\sum_{n\leq x}a^{+}(n)(x-n)^{\rho} +\frac{x^{\rho}}{2 \pi i } \sum_{n \le x} a^-(n) g_{\rho} \left(n, x \right) - Q_{\rho}(x).
    \end{equation}
We can also see that
\begin{equation}\label{diff}t'_{\rho+1}(y^2)=2y t_{\rho}(y^2)
\end{equation}
This is a direct computation, where, for the second term we use the integral expression for the series $\sum a^-(n) g_{\rho+1}(n, y^2)$ given in \eqref{g(n, x)}. Specifically, the second term in \eqref{diff} can be written as
$$y^{2\rho+2}\sum_{n \le y^2} a^-(n) g_{\rho+1}(n,y^2)= \int_{(\alpha)}\frac{W_{1-k}(s)L^{-}(f,s)}{\Gamma(\rho+2+s)}y^{2\rho+2+2s} \ ds$$
and differentiating with respect to $y$ gives
\begin{align*}
   2\int_{(\alpha)}\frac{(s + \rho+2) W_{1-k}(s)L^{-}(f,s)}{\Gamma(\rho+1+s)}y^{2s+ 2\rho+1}\ ds &= 2\int_{(\alpha)}\frac{ W_{1-k}(s)L^{-}(f,s)}{\Gamma(\rho+1+s)}y^{2s+ 2\rho+1}\ ds \\
   &=2y^{2\rho+1}\sum_{n \le y^2} a^-(n) g_{\rho}(n,y^2)
\end{align*}
Combining \eqref{deriv}, \eqref{deriv2} \eqref{diff}, we deduce 
$$\frac{1}{\Gamma(\rho+1)}\sum_{n\leq y^2}a^{+}(n)(y^2-n)^{\rho} +\frac{y^{2\rho}}{2 \pi i } \sum_{n \le y^2} a^-(n) g_{\rho} \left(n, y^2 \right) - Q_{\rho}(y^2)=A_{\rho}(y),$$
which (with $x=\sqrt{y}$) is Theorem \ref{thm:summationformula} for $\rho>\rho_0-\frac12.$
\end{proof}

%%%%%%%%%%%%%%%%%%%%%%%%%%%%%%%%%%%%%
%%%%%%%%%%%%%%%%%%%%%%%%%%%%%%%%%%%%%
%%%%%%%%%%%%%%%%%%%%%%%%%%%%%%%%%%%%%
\subsection{Asymptotic formula}
\label{sec:AsymptoticAnalysis}

We can now deduce our main asymptotic result from the summation formula proved above.
\begin{theorem}\label{th:SummationAsymptotic}
Assume all the notation in Theorem \ref{thm:summationformula}. For $\rho>\rho_0-\frac{1}{2}$,
we have
\begin{align}\label{eq:SummationAsymptotic}
\frac{1}{\Gamma(\rho+1)}\sum_{n\leq x}a^{+}(n)(x-n)^{\rho} &+\frac{x^{\rho}}{2 \pi i } \sum_{n \le x} a^-(n) g_{\rho} \left(n, x \right) \nonumber \\
&= Q_{\rho}(x) + O(x^{\frac{\rho}{2} + \max \{ \frac{1}{4}, \frac{2k-1}{4} \}} ) \nonumber \\
&= \begin{cases}
x^{\rho + k} \frac{b^{+}(0) i^k  (2 \pi)^k}{N^{\frac{k}{2}} \Gamma(\rho + k + 1)} + O(x^{\rho+1}) & 2 \ge k >1, \\
\frac{2\pi i}{\Gamma(\rho+2)} x^{\rho+1} \left(b^-(0)+ \frac{b^+(0)}{N} \right) + O(x^{\rho}) & k=1, \\
x^{\rho+1} \frac{2 \pi  N^{\frac{k}{2}-1} b^-(0) i^k}{\Gamma(\rho + 2)} + O(x^{\max \{ \rho, \rho+k
%, \frac{\rho}{2} + \frac{1}{4} 
\}})  & k < 1.
\end{cases}
\end{align}
\end{theorem}
\begin{proof} From Theorem \ref{thm:summationformula}, we see that \eqref{eq:sumformula} holds for $\rho>\rho_0-\frac12.$ Then, \eqref{series} allows us to estimate the series in the right-hand side of \eqref{eq:sumformula}, for $\rho>\rho_0-\frac12$ to deduce
\begin{equation}
\frac{1}{\Gamma(\rho+1)}\sum_{n\leq x}a^{+}(n)(x-n)^{\rho} +\frac{x^{\rho}}{2 \pi i } \sum_{n \le x} a^-(n) g_{\rho} \left(n, x \right)= Q_{\rho}(x)+O(x^{\frac{\rho +1}{2}-\frac14})+O(x^{\frac{\rho +k}{2}-\frac14})
\end{equation}
Again from our choice of $\rho$, we have $\rho+1>\rho+k>(\rho+1)/2-1/4>(\rho+k)/2-1/4$, if $k<1$ and 
$\rho+k \ge \rho+1 >\rho \ge (\rho+k)/2-1/4 \ge (\rho+1)/2-1/4$, if $k \ge 1$. Therefore, since $Q_{\rho}(x)$ is a linear combination of $x^{\rho},$ $x^{\rho+1}$, $x^{\rho+k-1}$ and $x^{\rho+k}$, we deduce \eqref{eq:SummationAsymptotic}.
\iffalse
\begin{align}
&\frac{1}{\Gamma(\rho+1)}\sum_{n\leq x}a^{+}(n)(x-n)^{\rho} +\frac{x^{\rho}}{2 \pi i } \sum_{n \le x} a^-(n) g \left(n, x \right) \nonumber\\
&\qquad= Q_{\rho}(x) - i^kx^{(\rho +k)/2}\left(\frac{\sqrt{N}}{2\pi }\right)^{\rho}\sum_{n=1}^{\infty}\frac{b^{+}(n)}{n^{(\rho+k)/2}}J_{\rho+k}\left(4\pi\sqrt{\frac{nx}{N}}\right)\nonumber\\
&\qquad -i^kx^{(\rho+1)/2}\left(\frac{\sqrt{N}}{2\pi}\right)^{\rho+k-1}\sum_{n=1}^{\infty} \frac{b^-(n)}{n^{(\rho-1)/2+k}}\int_{0}^{1/2}\frac{u^{k+(\rho-3)/2}}{(1-u)^{(1+\rho)/2}
}J_{\rho+1}\left(4\pi\sqrt{\frac{nx(1-u)}{Nu}}\right)du.
\end{align}
Here
$$
Q_{\rho}(x) = x^{\rho} \left( -\frac{a^+(0)}{\Gamma(\rho + 1)} + \frac{2 \pi x N^{\frac{k}{2}-1} b^-(0) i^k}{\Gamma(\rho + 2)} - \frac{ a^-(0) x^{k-1} ( 2 \pi )^{k-1}}{ \Gamma(\rho +  k)}  + \frac{b^{+}(0) i^k x^{k} (2 \pi)^k}{N^{\frac{k}{2}} \Gamma(\rho + k + 1)} \right)
$$

Taking the $m=0$ truncation of the Bessel $J$-function asymptotic formula \eqref{eq:ApproximatedJSeries}, we have  
\begin{align*}
    F_{\nu}(y,b_n,\mu_n) \sim \sum_{n\geq 1} b_n(y^2/\mu_n)^{\nu/2} J_{\nu}(4\pi\sqrt{\mu_n}y)<< y^{\nu -1/2} 
\end{align*}
as $y\to\infty $.

Considering the residue term, $Q_\rho(y) << y^{\max(\rho + k,\rho + 1)}$

For $F_{\nu}(y)$ and its integral transform, we have the bound $y^{\rho + k -1/2}$.

We see that the leading term will have order $x^{\max(\rho+k, \rho+1)}$, and we can similarly deduce the second leading term. 

\fi
\end{proof}

The lemma below gives an estimate for the second sum in the left-hand side of \eqref{eq:SummationAsymptotic}.
\begin{lemma}\label{suppl} Let $f \in H_k^{\text{Eis}} (N, \chi)$ have the expansion \eqref{FEf}. 
%(which holds if $\rho > \rho_0-1$). 
Then, for $\alpha > k-1$ such that $L_f^-(\alpha)$ is absolutely convergent and
such that $W_{1-k}(s)$ is analytic on $\operatorname{Re}(s) = \alpha$, for all $\rho > 1-k$ we have
$$
\sum_{n \le x} a^-(n) g_{\rho} \left(n, x \right) = O ( x^{\alpha}) \qquad \text{as $x \to \infty$.}$$
\end{lemma}
\begin{proof} With the assumptions of the lemma, Theorem \ref{thm:perrongen} applies to give
\begin{align*}
\left | \sum_{n \le x} a^-(n) g_{\rho} \left(n, x \right)\right |= \left |\int_{(\alpha)} \frac{L^-(f, s) W_{1-k}(s) x^{s}}{\Gamma(\rho + 1 + s)} ds \right | \le x^{\alpha}\int_{(\alpha)} \frac{ |L^-(s) W_{1-k}(s)|}{|\Gamma (\rho + 1 + s)|} ds 
\end{align*}
Stirling's bound and Corollary \ref{thm:Wbound} show that the integral is $O(1)$.

\end{proof}

%%%%%%%%%%%
\section{Refining Theorem \ref{thm:summationformula}}\label{extendingrange}
In this section we explore various refinements of Theorem \ref{thm:summationformula}.  First we present two alternative versions of Theorem \ref{thm:summationformula} that may be more appealing. Then we give a version in the $k = \frac{3}{2}$ case with an extended range for $\rho$. 

Following the notation of \cite{ChandNaras61}, we assign to the sequences $(b_n)_n$ and $(\mu_n)_n$ the series $F_{\nu}\left(y ; b_n; \mu_n \right)$:
\begin{equation}\label{eq:Fnu}
F_{\nu} \left(y ; b_n; \mu_n \right)  := \sum_{n=1}^{\infty} b_n \left(\frac{y^2}{\mu_n} \right)^{\frac{\nu}{2}}  J_{\nu} \left( 4 \pi \sqrt{\mu_n}  y \right).
\end{equation}
First we show that these functions are central to understanding both infinite sums in Theorem \ref{thm:summationformula}. 

When the sum defining $F_{\rho+1}(y, \frac{b^-(n)}{n^{k-1}}, n)$ converges absolutely, we can express the right-hand side of our summation formula in terms of $F_{\rho +k}$ and $F_{\rho + 1}$, as follows:
\begin{lemma}\label{thm:integral-transform}
When $F_{\rho+1}(y, \frac{b^-(n)}{n^{k-1}}, n)$ converges absolutely and $\rho > \frac{1}{2} - 2k$, we have
\begin{align*}
A_{\rho}(y) =-i^k  \frac{N^{-\frac{k}{2}}}{(2 \pi)^{\rho}}  F_{\rho + k} \left( y; b^+(n); \frac{n}{N} \right) - i^k \frac{N^{\frac{k}{2}-1}}{(2 \pi)^{\rho + k -1}} \int_0^{\frac{1}{2}} u^{k-2} F_{\rho+1} \left( y; \frac{b^{-}(n)}{n^{k-1}} ; \frac{n(1-u)}{Nu} \right) du,
\end{align*}
with $A_{\rho}(y)$ as defined in \eqref{eq:ArhoDef}. 
\end{lemma}
\begin{remark}\label{SeSt}
In general, $F_{\rho+1}(y, \frac{b^-(n)}{n^{k-1}}, n)$ converges absolutely for $\rho > \rho_0 - \frac{1}{2}$, as used in the proof of Theorem \ref{thm:summationformula}. When $k=\frac{3}{2}$, $F_{\rho+1}(y, \frac{b^-(n)}{n^{k-1}}, n)$ converges absolutely for $\rho > 2 \mu_g^- -2k + \frac{3}{2}$, extending the range of $\rho > \rho_0 - \frac{1}{2}$. This is because the sequences $a^-(n)$ and $b^-(n)$ are supported in a set of the form $\{t_im^2; m \in \mathbb N, i=1, \dots \ell\}$. Indeed, they are Fourier coefficients of $\xi_{3/2}(f)$ and $\xi_{3/2}(g)$, which are modular forms of weight $1/2$ and hence, by Serre-Stark's Basis Theorem \cite{SSt}, they are a linear combination of 
%classes of 
theta series. 
%When $k=\frac{3}{2}$, it follows form the Serre-Stark Basis Theorem that the $b^-(n)$ are supported on only finitely many square classes, which implies that $F_{\rho+1}(y, \frac{b^-(n)}{n^{k-1}}, n)$ converges absolutely for $\rho > 2 \mu_g^- -2k + \frac{3}{2}$, extending the range of $\rho > \rho_0 - \frac{1}{2}$.
\end{remark}
\begin{proof}
We make the change of variables $v = \frac{1}{u} - 1$ to rewrite the integral on the right as
\begin{align}\label{chvar}
\int_0^{\frac{1}{2}} u^{k-2} F_{\rho+1} \left( y; \frac{b^{-}(n)}{n^{k-1}} ; \frac{n(1-u)}{Nu} \right) du 
%&= \int_{\infty}^1 (v+1)^{-k} F_{\rho+1} \left( y; \frac{b^{-}(n)}{n^{k-1}} ; \frac{nv}{N} \right) (-dv) \nonumber \\&
= \int_1^{\infty} (v+1)^{-k} F_{\rho+1} \left( y; \frac{b^{-}(n)}{n^{k-1}} ; \frac{nv}{N} \right) dv.
\end{align}
Let $|F_{\rho+1}|$ denote the series of the absolute values of the terms of $F_{\rho+1}$. 
By Proposition \ref{thm:besselj-asymptotics}, $|J_{\rho+1}(4 \pi y \sqrt{nv/N})| \ll v^{-1/4} |J_{\rho+1}(4 \pi y \sqrt{n})|$ as $n \to \infty$, if $v \ge 1$ and hence, the integrand in \eqref{chvar} is bounded by
$$
|(v+1)^{-k} F_{\rho+1} \left( y; \frac{b^{-}(n)}{n^{k-1}} ; \frac{nv}{N} \right) | \ll |F_{\rho+1}| \left( y; \frac{b^{-}(n)}{n^{k-1}} ; n \right)  (v+1)^{-k } v^{-(\frac{\rho+1}{2}) - \frac{1}{4}}
$$
for $v \in (1,\infty)$. If $-k - \frac{(\rho+1)}{2} - \frac{1}{4} < -1$, the integral converges absolutely and by Fubini's Theorem, we can swap the sum and the integral, producing the second sum in the formula for $A_{\rho}(y)$.
\end{proof}

The previous lemma lets us rewrite the summation formula in Theorem \ref{thm:summationformula} more succinctly:
\begin{corollary}\label{thm:refined-summation}
Assume the notation and assumptions of Theorem \ref{thm:summationformula}. For $\rho > \rho_0 - \frac{1}{2}$, we have
\begin{align} \label{eq:refined-summation}
&\frac{1}{\Gamma(\rho+1)}\sum_{n\leq y^2}a^{+}(n)(y^2-n)^{\rho} +\frac{x^{\rho}}{2 \pi i } \sum_{n \le y^2} a^-(n) g_{\rho}\left(n, y^2 \right) - Q_{\rho}(y^2) \nonumber\\
%&=  - i^k  \frac{N^{-\frac{k}{2}}}{(2 \pi)^{\rho}}  F_{\rho + k} \left( \sqrt{x}; b^+(n); \frac{n}{N} \right) - i^k \frac{N^{\frac{k}{2}-1}}{(2 \pi)^{\rho + k -1}} \int_0^{\frac{1}{2}} u^{k-2} F_{\rho+1} \left( \sqrt{x}; \frac{b^{-}(n)}{n^{k-1}} ; \frac{n(1-u)}{Nu} \right) du
&=-i^k  \frac{N^{-\frac{k}{2}}}{(2 \pi)^{\rho}}  F_{\rho + k} \left( y; b^+(n); \frac{n}{N} \right) - i^k \frac{N^{\frac{k}{2}-1}}{(2 \pi)^{\rho + k -1}} \int_0^{\frac{1}{2}} u^{k-2} F_{\rho+1} \left( y; \frac{b^{-}(n)}{n^{k-1}} ; \frac{n(1-u)}{Nu} \right) du.
\end{align}
\end{corollary}

We can realize the integrand on the second line in terms of the ``shadow" of $g=f|w_N$. To be more specific, let $g_1 := - (4\pi)^{k-1} \xi_k(g)$. Then Theorem 5.9 of \cite{thebook} can be used to compute 
$$
g_1 = (k-1) (4 \pi)^{k-1} \overline{b^-(0)} + \sum_{n=1}^{\infty} \frac{\overline{b^{-}(n)}}{n^{k-1}} q^n \in M_{2-k}(N, \chi).
$$ 
Let $f_1 = g_1 | w_N $. Then \eqref{eq:intertwining-property} shows
\begin{align*}
f_1 &= -(4\pi)^{k-1} \xi_k(g)|w_N \\
&= - (Ni)^{-1} (4\pi)^{k-1} \xi_k(g|w_N) \\
&= N^{-1} i (4 \pi)^{k-1} \xi_k( i^{-2k} f) \\
&= N^{-1} i^{2k + 1} (4 \pi)^{k-1} \xi_k(  f). 
\end{align*}
Again using Theorem 5.9 of \cite{thebook},
\begin{align*}
f_1 = (1-k) N^{-1} i^{2k+1} (4 \pi)^{k-1} \overline{a^-(0)}- N^{-1} i^{2k+1} \sum_{n=1}^{\infty} \frac{\overline{a^{-}(n)}}{n^{k-1}} q^n \in M_{2-k}(N, \overline{\chi}).
%\\&=: a'(0) + \sum_{n=1}^{\infty} \frac{a'(n)}{n^{k-1}} q^n \in M_{2-k}(N, \overline{\chi})
\end{align*}
\begin{lemma}\label{thm:shadow}
Let $\beta$ be such that $\sum_n |b^-(n)|n^{1-k-\beta}<\infty.$ For $\rho > 2 \beta - 
%\frac{5}{2},
\frac{3}{2},$ $F_{\rho+1}\left (y,\frac{\overline{b^-(n)}}{n^{k-1}},\frac{n v}{\sqrt{N}} \right )$ is piecewise continuous as a function in $v \in (0,\infty)$, and 
\begin{equation}\label{eq:shadow-relation}
\frac{i^{k-2} }{(2 \pi)^{\rho+k-1} } F_{\rho+1} \left (y, \frac{\overline{b^-(n)}}{n^{k-1}}, 
\frac{nv}{\sqrt N} \right) =\frac{1}{\Gamma (\rho +k)} \sum_{n \le v \sqrt{N}y^2} \frac{\overline{a^-(n)}}{n^{k-1}v^{2-k}}  \left(y^2 - \frac{n}{v\sqrt{N}}\right)^{\rho + k -1}-Q_{\rho + k -1}(y^2,v)
\end{equation}
where 
$$
Q_{\rho+k-1}(x,v) = (k-1)(4\pi)^{k-1} x^{\rho+ k -1} v^{k-2}\left(\frac{i^{2k+1} N^{-1} \overline{a^-(0)}}{\Gamma(\rho + k)} + 
\frac{(2 \pi xv i)^{2-k}\overline{b^{-}(0)}}{\Gamma(\rho+2)} \right)
$$
Moreover, $F_{\rho + 1}(y,b_n, n v)=O(v^{\max \{ \mu_f^-, k-1 \}})$, as $v \to \infty$. 

If, in addition to $\sum_n |b^-(n)|n^{1-k-\beta}<\infty$, we have 
$\sup_{0 \le h \le 1} \left |\sum_{m^2 < n < (m+h)^2} b^-(n) n^{\frac{3}{2}-\beta-k} \right |= o(1),$ then, \eqref{eq:shadow-relation} holds for $\rho > 2 \beta - \frac{5}{2}$.
\end{lemma}
\begin{proof}
The modular forms $g_1$ and $f_1$ have corresponding Dirichlet series
\begin{equation}\label{Defpsi}
\psi(s) = i^{k-2}\sum_{n=1}^{\infty} \frac{\overline{b^-(n)}\sqrt{N}^s}{n^{k-1+s} v^s} =i^{k-2}\sqrt{N}^s v^{-s} L(g_1,s)=i^{k-2}\frac{(2\pi)^s}{\Gamma(s)}v^{-s}\Lambda(g_1, s)
\end{equation}
and
\begin{equation}\label{Defphi}
\varphi(s) = \sum_{n=1}^{\infty} \frac{\overline{a^{-}(n)} v^{k-2}}{n^{k-1}(n/v)^{s}}\sqrt{N}^s = v^{k-2+s}\sqrt{N}^s L(f_1, s)
=\frac{(2\pi)^s}{\Gamma(s)}v^{k-2+s}\Lambda(f_1, s).
\end{equation}
Since $f_1=g_1|w_N$, Theorem \ref{thm:ShankadharSingh} implies $\Lambda(g_1, s)=i^{2-k}\Lambda(f_1, 2-k-s)$, or $\Lambda(f_1, s)=i^{k-2}\Lambda(g_1, 2-k-s)$ and hence
$$\Gamma(s)(2\pi)^{-s}\varphi(s)=v^{k-2+s}i^{k-2}\Lambda(g_1, 2-k-s)=\Gamma(2-k-s)(2\pi)^{-2+k+s}\psi(2-k-s).$$ Further, $\Lambda(f_1,s)$ has a meromorphic continuation to $\mathbb{C}$ with poles at $0$ and $2-k$ (the other possible poles described in Theorem \ref{thm:ShankadharSingh} do not occur here since $f_1$ is holomorphic) with residues equal to $(k-1)N^{-1}i^{2k+1}(4 \pi)^{k-1} \overline{a^{-}(0)}, (k-1)i^{2-k}(4\pi)^{k-1}\overline{b^{-}(0)}$. Therefore, Lemma 5 of \cite{ChandNaras61} applies with the $\varphi$ and $\psi$ defined in \eqref{Defphi} and \eqref{Defphi}, $a_n = \overline{a^{-}(n)} n^{1-k} v^{k-2}$, $\lambda_n = n/(v \sqrt{N})$, $b_n = \frac{i^{k-2}\overline{b^-(n)}}{n^{k-1}}$, $\mu_n = nv/\sqrt{N}$, $\delta = 2-k$. 
We deduce \eqref{eq:shadow-relation} for $\rho > 2 \beta-\frac32$. Note that we have applied Lemma 5 of \cite{ChandNaras61} with $\rho+k-1$ in place of $\rho.$ Finally, applying Theorem IV of \cite{ChandNaras61}, if the additional condition on $b^-(n)$ holds, we deduce that \eqref{eq:shadow-relation} holds on the extended range $\rho > 2\beta - \frac{5}{2}$.

    As $v$ increases, the sum on the right-hand side of \eqref{eq:shadow-relation} is piecewise continuous with jumps at $y^{-2} N^{-\frac{1}{2}} \mathbb{Z}$, when $\rho+k-1<0$. The terms are of order $O(v^{k-2} \max\{ \overline{a^{-}(n)}n^{1-k} : n \le v \sqrt{N} y^2 \} )$ and the number of terms is $\lfloor v \sqrt{N}y^2 \rfloor$. We conclude that the sum is of order $O(v^{\max \{ \mu_f^-, k-1 \}})$. The order of growth of $Q_{\rho+k-1}(x)$ as $v$ increases is $O(v^{\max\{ k-2, 0\}})$.
\end{proof}

Once we conjugate \eqref{eq:shadow-relation} and replace $v$ with $v/\sqrt{N}$, we obtain another version of Theorem \ref{thm:summationformula} in which the integrand in Corollary \ref{thm:refined-summation} is made elementary.
\begin{corollary}\label{thm:refined-summation-2}
Assume the notation and assumptions of Theorem \ref{thm:summationformula}. For $\rho > \rho_0 - \frac{1}{2}$, we have
\begin{align} \label{eq:refined-summation-2}
&\frac{1}{\Gamma(\rho+1)}\sum_{n\leq x}a^{+}(n)(x-n)^{\rho} +\frac{x^{\rho}}{2 \pi i } \sum_{n \le x} a^-(n) g_{\rho}\left(n, x \right) - Q_{\rho}(x) \nonumber\\
&=-i^k  \frac{N^{-\frac{k}{2}}}{(2 \pi)^{\rho}}  F_{\rho + k} \left( y; b^+(n); \frac{n}{N} \right) - i^{2k}  N^{\frac{k-2}{2}} \int_1^{\infty} (v+1)^{-k}  \overline{Q_{\rho + k -1}\left (y^2,\frac{v}{\sqrt{N}} \right )} dv  \nonumber \\
&+\frac{i^{2k} }{\Gamma(\rho+k)} \int_1^{\infty} (v+1)^{-k}  \sum_{n \le v y^2} \frac{a^{-}(n)}{n^{k-1}} v^{k-2} \left(y^2 - \frac{n}{v}\right)^{\rho + k -1} dv.
\end{align}

\iffalse
\begin{align} \label{eq:refined-summation-2}
&\frac{1}{\Gamma(\rho+1)}\sum_{n\leq x}a^{+}(n)(x-n)^{\rho} +\frac{x^{\rho}}{2 \pi i } \sum_{n \le x} a^-(n) g_{\rho}\left(n, x \right) - Q_{\rho}(x) \nonumber\\
%&=  - i^k  \frac{N^{-\frac{k}{2}}}{(2 \pi)^{\rho}}  F_{\rho + k} \left( \sqrt{x}; b^+(n); \frac{n}{N} \right) - i^k \frac{N^{\frac{k}{2}-1}}{(2 \pi)^{\rho + k -1}} \int_0^{\frac{1}{2}} u^{k-2} F_{\rho+1} \left( \sqrt{x}; \frac{b^{-}(n)}{n^{k-1}} ; \frac{n(1-u)}{Nu} \right) du
&=-i^k  \frac{N^{-\frac{k}{2}}}{(2 \pi)^{\rho}}  F_{\rho + k} \left( y; b^+(n); \frac{n}{N} \right) + i^k  N^{\frac{k}{2}-1} \int_1^{\infty} (v+1)^{-k}  Q_{\rho + k -1}(y^2,v) dv  \nonumber \\
&- \frac{i^k N^{\frac{k}{2}-1}}{\Gamma(\rho+k)} \int_1^{\infty} (v+1)^{-k}  \sum_{n \le v \sqrt{N}y^2} \frac{\overline{a^{-}(n)}}{n^{k-1}} v^{k-2} \left(y^2 - \frac{n}{v\sqrt{N}}\right)^{\rho + k -1} dv
\end{align}
\fi
\end{corollary}
When $\mu_g^- - k < -1$, we can rewrite the last term as
\begin{align*}
\frac{i^{2k}}{\Gamma(\rho+k)}\sum_{n=1}^{\infty} a^{-}(n) \int_{\frac{1}{y^2 }}^{\infty} (nv+1)^{-k} v^{k-2} \left(y^2 - \frac{1}{v}\right)^{\rho + k -1} dv.
\end{align*}
\iffalse
\begin{align*}
&\frac{i^k N^{\frac{k}{2}-1}}{\Gamma(\rho+k)} \int_1^{\infty} (v+1)^{-k}  \sum_{n \le v \sqrt{N}y^2} \frac{\overline{a^{-}(n)}}{n^{k-1}} v^{k-2} \left(y^2 - \frac{n}{v\sqrt{N}}\right)^{\rho + k -1} dv \\
&= \frac{i^k N^{\frac{k}{2}-1}}{\Gamma(\rho+k)} \sum_{n=1}^{\infty} \frac{\overline{a^{-}(n)}}{n^{k-1}} \int_{\frac{n}{y^2 \sqrt{N}}}^{\infty} (v+1)^{-k} v^{k-2} \left(y^2 - \frac{n}{v\sqrt{N}}\right)^{\rho + k -1} dv\\
&= \frac{i^k N^{\frac{k}{2}-1}}{\Gamma(\rho+k)} \sum_{n=1}^{\infty} \overline{a^{-}(n)} \int_{\frac{1}{y^2 \sqrt{N}}}^{\infty} (nv+1)^{-k} v^{k-2} \left(y^2 - \frac{1}{v\sqrt{N}}\right)^{\rho + k -1} dv\\
\end{align*}
\fi

\subsection{Extending the range for $\rho$}
Analogously to Theorems III and IV of \cite{ChandNaras61}, we prove that the range for $\rho$ in Theorem \ref{thm:summationformula} can be extended to $\rho > \rho_0 - \frac{3}{2}$, when $k = \frac{3}{2}$. In this range, the $F_{\rho+1}$ sum converges absolutely, while the $F_{\rho + k}$ does not converge absolutely but converges conditionally by the work of \cite{ChandNaras61}.

Throughout this section, $I$ is an open interval of length $1$, and $J$ is a closed interval inside $I$. Following \cite{ChandNaras61} we let $S[f]$ denote the Fourier series $\sum_{n \in \mathbb{Z}} s_n e^{ i n x}$ of a continuous function with period $1$ that coincides with $f$ on $J$. Further, we let $\lambda(x)$ be a smooth function with compact support on $I$ such that $\lambda(x)=1$ for $x\in J$. Finally, we let $\equiv$ denote equiconvergence on $I$, i.e. $\sum_{n=1}^{\infty} a_n(x) \equiv \sum_{n=1}^{\infty} b_n(x)$ means that $\sum_{n=1}^{\infty}(a_n(x) - b_n(x))$ converges uniformly on $I$.

A general lemma that we will exploit is the following.
\begin{lemma}\label{thm:uniform-convergence}
    Assume $f(x)$ is continuous and differentiable on $I$. If $f'(x)$ is piecewise continuous on $I$, then $S[\lambda(x) f(x)]$ converges uniformly on $I$. If $f'(x)$ is differentiable on $I$ and $f''(x)$ is piecewise continuous on $I$, then $S'[\lambda(x) f(x)]$ converges uniformly on $I$.
\end{lemma}
\begin{proof}
Since $f(x)$ and $\lambda(x)$ are both continuous and differentiable, $f(x) \lambda (x)$ is continuous and differentiable on $I$ with continuous derivative, so the uniform convergence of the Fourier series of $f(x) \lambda(x)$ follows from Theorem 13.7 of \cite{Howell}. 

 Similarly, our assumptions ensure $(f\lambda)' (x) = f'(x) \lambda(x) + f(x) \lambda'(x)$ is continuous and differentiable on $I$ with piecewise continuous derivative (it could be discontinuous at the endpoints of $I$). So, by Theorem 13.7 of \cite{Howell} 
 $S'[\lambda(x) f(x)]=S[(\lambda f)'(x)]$ converges uniformly on $I$.
 %, and $S[(f\lambda')]$ is given by $S'[f(x)]$. 
\end{proof}

By Theorem II of \cite{ChandNaras61}, we have that if 
$
\sum \frac{|b_n|}{n^{\frac{\nu}{2} + \frac{3}{4}}} < \infty
$
and
$
\sup_{0 \le h \le 1} \left |\sum_{m^2 < n < (m+h)^2} \frac{b_n}{n^{\frac{\nu}{2} + \frac{1}{4}}} \right |= o(1),
$
then we have, on $J$,
\begin{equation}\label{ThII}
S'[\lambda(y) F_{\nu+1} \left(y ; b_n; \mu_n \right) ] \equiv F_{\nu} \left(y ; b_n; \mu_n \right).
\end{equation} 
Formally, we wish to differentiate \eqref{eq:refined-summation}, but we have to be careful because of the conditional convergence of the $F_{\rho+k}$ sum on this range. For convenience, we set
\begin{align*}
s_{\rho}(y) &:= t_{\rho}(y^2)
 %\frac{1}{\Gamma(\rho+1)}\sum_{n\leq y^2}a^{+}(n)(y^2-n)^{\rho} +\frac{y^{2\rho}}{2 \pi i } \sum_{n \le y^2} a^-(n) g_{\rho}\left(n, y^2 \right) - Q_{\rho}(y^2)  \\ &
 + i^k \frac{N^{\frac{k}{2}-1}}{(2 \pi)^{\rho + k -1}} \int_0^{\frac{1}{2}} u^{k-2} F_{\rho+1} \left( y; \frac{b^{-}(n)}{n^{k-1}} ; \frac{n(1-u)}{Nu} \right) du,
\end{align*}
where $t_{\rho}(x)$ is as defined in \eqref{trho}.
Then \eqref{eq:refined-summation} says that for $\rho > \rho_0 - \frac{1}{2}$, we have
$$
s_{\rho} (y) =-i^k  \frac{N^{-\frac{k}{2}}}{(2 \pi)^{\rho}}  F_{\rho + k} \left( y; b^+(n); \frac{n}{N} \right).
$$
\begin{theorem}\label{thm:extendedsummation}
Assume the same conditions on $f,g$ as in Theorem \ref{thm:summationformula} with $k = \frac{3}{2}$.  Then we have \eqref{eq:refined-summation} for $\rho > \rho_0 - \frac{3}{2}$.
\end{theorem}
\begin{proof}
%We first observe that the sequences $a^-(n)$ and $b^-(n)$ are supported in a set of the form $\{t_im^2; m \in \mathbb N, i=1, \dot \ell\}$. This is because $\xi_{3/2}(f), \xi_{3/2}(g)$ are modular forms of weight $1/2$ and hence, by Serre-Stark's Basis Theorem, they are a linear combination of classes of theta series. 
{\bf Claim 1:} We have $s_{\rho+1}'(y) = 2y s_{\rho}(y)$. Since $\rho+1>0,$ $t_{\rho}(y^2)$ is differentiable in $\mathbb R_+$, as seen in the proof of Theorem \ref{thm:summationformula}, and, with \eqref{diff}, $t'_{\rho+1}(y^2)=2y t_{\rho}(y^2)$. For the remaining term of $s_{\rho}(y)$,
%We only need to check that this relation holds for the third term of $t_{\rho}(y)$. First 
we note that for any $u$, 
%the uniform convergence and continuity of 
$F_{\rho+1}(y):=F_{\rho+1} \left( y; \frac{b^{-}(n)}{n^{k-1}} ; \frac{n(1-u)}{Nu} \right)$ is uniformly convergent and thus continuous. Indeed, we first apply Lemma \ref{thm:shadow} with $k=3/2$, with $\beta=\mu^-_g+\varepsilon$. Because of Remark \ref{SeSt}, $\sum_n |b^-(n)|n^{1-k-\beta}$ will then converge and hence \eqref{eq:shadow-relation} will hold for $\rho>2\mu_g^--\frac{3}{2}.$
This, in turn, implies that Theorem III of \cite{ChandNaras61} is applicable
with $\delta=1/2$. 
\iffalse
\footnote{Here we need that $\sup_{0 \le h \le 1} \left |\sum_{Nm^2 < n < N(m+h)^2} \frac{b^-(n)}{n^{k-1+\beta-\frac{1}{2}}} \right |= o(1).$ This holds because for $\beta=\mu^-_g+\varepsilon$ (with Remark \ref{SeSt}) 
$$\left |\sum_{Nm^2 < n^2 < N(m+h)^2} \frac{b^-(n^2)}{n^{2(k-1)+2\beta-1}} \right | \le \sum_{\sqrt{N}m < n < \sqrt{N}(m+h)} \frac{1}{n^{2 \varepsilon}} \le \int_{\sqrt{N}m}^{\sqrt{N}(m+h)}x^{-2\varepsilon}dx \ll (m+h)^{1-2\varepsilon}-m^{1-2\varepsilon}$$
For $h \in [0, 1]$, this is $\le (m+1)^{1-2\varepsilon}-m^{1-2\varepsilon}$
which is $o(1)$ as $m \to \infty.$}
\fi
Hence
$F_{\rho+1}(y)$ converges uniformly for $\rho+k-1=\rho+\frac{1}{2} \ge 2(\mu^-_g+\varepsilon)-\frac{1}{2}-\frac{3}{2}=2 \mu_g^--2$. This is satisfied when $\rho>\rho_0-\frac{3}{2}=2 \mu_g^--1$.
%For this to hold, Th. III is also asking for the LHS of (46) to be continuous. This is indeed the case here because $\rho+\frac{1}{2}>2\mu_g^--\frac12>0.$ (or $\rho>0$).
We further observe that the term-by-term derivative of $F_{\rho+2} \left( y
%; \frac{b^{-}(n)}{n^{k-1}}; \frac{n(1-u)}{Nu} 
\right)$  is $2 \pi \cdot 2y F_{\rho+1} \left( y
%; \frac{b^{-}(n)}{n^{k-1}} ; \frac{n(1-u)}{Nu} 
\right)$. Since we just saw that 
$F_{\rho+1} \left( y
%; \frac{b^{-}(n)}{n^{k-1}} ; \frac{n(1-u)}{Nu} 
\right)$ converges uniformly, we can interchange summation and differentiation to deduce $(2 \pi)^{-\rho-k}F'_{\rho+2} \left( y
%; \frac{b^{-}(n)}{n^{k-1}} ; \frac{n(1-u)}{Nu} 
\right)=2y(2 \pi)^{-\rho-k+1}F_{\rho+1} \left( y
%; \frac{b^{-}(n)}{n^{k-1}} ; \frac{n(1-u)}{Nu} 
\right).$
Differentiating under the integral sign of the last term of $s_{\rho(y)}$ proves the claim. 

{\bf Claim 2:} The function $s_{\rho+1}'(y)$ is piecewise smooth (i.e. has a derivative which is piecewise continuous).
By Claim 1, we just need to check that $s_{\rho}(y)$ is smooth. 
Using Lemma \ref{thm:shadow}, we can rewrite $s_{\rho}(y)$ as
\begin{align*}
&t_{\rho}(y^2)+
%(2 \pi)^{k+\rho-1}  i^k 
(-1)^{k} N^{\frac{k-2}{2}}
%\frac{-N^{\frac{k}{2}-1}}{(2 \pi)^{\rho + k -1}} 
\int_1^{\infty} (v+1)^{-k} \overline{Q_{\rho + k -1}(y^2,\frac{v}{\sqrt N})}dv \\
 &-(-1)^k 
 %-i^k N^{\frac{k}{2}-1} 
 \int_1^{\infty} (v+1)^{-k}  \frac{1}{\Gamma (\rho +k)} \sum_{n \le v y^2} \frac{a^-(n)}{n^{k-1}} v^{k-2} \left(y^2 - \frac{n}{v}\right)^{\rho + k -1} dv.
\end{align*}
%The expressions on the first two lines 
The first two terms can be directly seen to have piecewise continuous derivatives.
% Since $\rho>0,$ $t_{\rho}(y^2)$ is differentiable in $\mathbb R_+$ and its derivative is continuous except for $y^2 \in \mathbb Z$.
On the other hand, $\rho+k-1>0,$ and, by Serre-Stark, we see, as in Remark \ref{SeSt}, that $\mu_f^{-} \le 1/2$. Therefore, with Fubini's theorem, we can interchange summation and integration in the last term to rewrite it (up to a constant) as
%We can rewrite the integral as
%$$-   i^k N^{\frac{k}{2}-1} \frac{1}{\Gamma (\rho +k)} \sum_{n=1}^{\infty} \frac{\overline{a^{-}(n)}}{n^{k-1}}  \int_{n/y^2 \sqrt{N}}^{\infty} (v+1)^{-k} v^{k-2} \left(y^2 - \frac{n}{v}\right)^{\rho + k -1} dv $$
\begin{equation}
    \label{termbyterm}
\sum_{n=1}^{\infty}\frac{a^{-}(n)}{n^{k-1}}  \int_{n/y^2}^{\infty} (v+1)^{-k} v^{k-2} \left(y^2 - \frac{n}{v}\right)^{\rho + k -1} dv. \end{equation}
We compute the term-by-term derivative of this sum to obtain, for $\rho + k -2 >-1$ %(always true for $\rho>0$, $k= \frac{3}{2}$),
$$
\sum_{n=1}^{\infty} a^{-}(n) (\rho + k -1) 2y \int_{\frac{1}{y^2}}^{\infty} (vn + 1)^{-k} v^{k-2} (y^2 - \frac{1}{v})^{\rho + k -2} dv. 
$$
This is bounded in absolute value by
$$
\sum_{n=1}^{\infty} |a^{-}(n) n^{-k}| (\rho + k -1) 2y
%_1
\int_{\frac{1}{y
%_1
^2 }}^{\infty} v^{-2} (y
%_1
^2 - \frac{1}{v})^{\rho + k - 2} dv.
$$
%for any $y_1$ above $I$. 
Since, as a function of $y$ this is uniformly bounded in any closed interval $J \subset I$, and by Remark \ref{SeSt} on the support of $(a^-(n))$ in classes of squares, we deduce absolute and uniform convergence
%This sum is convergent 
since $2k-\mu_f^{-}>1$. Therefore, \eqref{termbyterm} is differentiable with a continuous derivative.

Finally, we use Lemma \ref{thm:uniform-convergence} to conclude the result. 
Specifically, by Lemma \ref{thm:uniform-convergence} and Claim 2, we have that $S'[\lambda (y)s_{\rho+1}(y)]$ is uniformly convergent. 
Since $\rho + 1 > \rho_0- \frac{1}{2}$, Corollary \ref{thm:refined-summation} implies that 
$s_{\rho+1} (y) =-i^k  \frac{N^{-\frac{k}{2}}}{(2 \pi)^{\rho+1}}  F_{\rho + k+1} \left( y; b^+(n); \frac{n}{N} \right).$
%$t_{\rho+1}(y) = F_{\rho+k + 1}(y, b^+(n), \frac{n}{N})$. 
Thus,  $S'[\lambda (y) F_{\rho+k + 1}(y, b^+(n), \frac{n}{N})]$ is uniformly convergent. Further, by \eqref{ThII},
%Theorem II of \cite{ChandNaras61}, 
$F_{\rho+k}(y, b^+(n), \frac{n}{N}) \equiv S'[\lambda (y) F_{\rho+k + 1}(y, b^+(n), \frac{n}{N})].$
%. Since $S'[F_{\rho+k + 1}(y, b^+(n), \frac{n}{N})]$ is uniformly convergent
Indeed, for $\rho>\rho_0-\frac{3}{2},$ $\sum \frac{|b^+(n)|}{n^{\frac{\nu}{2} + \frac{3}{4}}} < \infty$ and
$
\sup_{0 \le h \le 1} \left |\sum_{Nm^2 < n < N(m+h)^2} \frac{b^+(n)}{n^{\frac{\rho+k}{2} + \frac{1}{4}}} \right |= o(1).
$
\iffalse
\footnote{To see the last assertion, notice that, if $\rho>\rho_0-\frac{3}{2}=2 \mu^+_b-1$, then
$$\left |\sum_{Nm^2 < n < N(m+h)^2} \frac{b^+(n)}{n^{\frac{\rho+k}{2} + \frac{1}{4}}} \right | \le
\sum_{Nm^2 < n < N(m+h)^2} \frac{1}{n^{\frac{\rho+k}{2} + \frac{1}{4}-\mu^+_g}} \le
\sum_{Nm^2 < n < N(m+h)^2} \frac{1}{n^{\frac{1}{2}}} \ll (m+h)^{\frac{1}{2}}-m^{\frac{1}{2}}. 
$$ This is $\le (m+1)^{\frac{1}{2}}-m^{\frac{1}{2}}$ for $h \in [0, 1]$, which has limit $0$, as $m \to \infty.$}
\fi

It follows that $ F_{\rho+k}(y, b^+(n), \frac{n}{N})$ is uniformly convergent and continuous. 
This, in particular, implies that $F_{\rho+k + 1}(y, b^+(n), \frac{n}{N})$ is term-by-term differentiable, because the term-by-term derivative is $2 \pi \cdot 2y F_{\rho+k}(y, b^+(n), \frac{n}{N}) $. Thus, $F_{\rho+k + 1}'(y, b^+(n), \frac{n}{N}) = 2 \pi \cdot 2y F_{\rho+k}(y, b^+(n), \frac{n}{N})$ and, in combination with Claim 1, we deduce 
$$s_{\rho} (y) = \frac{s'_{\rho+1} (y)}{2y}=\frac{-i^k N^{-\frac{k}{2}}}{2y(2 \pi)^{\rho+1}}F'_{\rho + k+1} \left( y; b^+(n); \frac{n}{N} \right )=-\frac{i^k N^{-\frac{k}{2}}}{(2 \pi)^{\rho}}F'_{\rho + k} \left( y; b^+(n); \frac{n}{N} \right )$$
as required.
\end{proof}

%%%%%%%%%%%%%%%%%%%%
\section{Applications}\label{application}
\subsection{Hurwitz Class Numbers}\label{Hurwitz}
For any discriminant $d<0$, let $\mathcal{Q}_d$ be the set of binary quadratic forms of discriminant $d$ which are not negative definite. The Hurwitz class numbers count $\SL_2(\mathbb{Z})$-classes of binary quadratic forms inversely weighted by stabilizer size: 
\begin{equation}\label{eq:Hurwitz-defn}
    H(n) := \!\!\!\!\!\!\!\!\!\! \sum_{Q \in \SL_2(\mathbb{Z}) \backslash \mathcal{Q}_{-n}} \frac{2}{|\operatorname{Stab} (Q)|},
\end{equation}
with the convention that $H(0) =  \frac{-1}{12}$ and $H(n)=0$ if $-n$ is neither zero nor a negative discriminant.

We have the famous result of Zagier from 1975, reformulated in the language of this paper. 
\begin{theorem}[Zagier \cite{zagier75}]
The function
\begin{equation}\label{eq:Zagier-EisensteinSeries}
\mathcal{H}(\tau) := - \frac{1}{12} + \sum_{n\geq 1}H(n)q^n + \frac{1}{8\pi \sqrt{v}} + \frac{1}{4\sqrt{\pi}}\sum_{n\geq 1} n \Gamma \left(- \frac{1}{2},4\pi n^2v \right)q^{-n^2}
\end{equation}
belongs to $H_{\frac{3}{2}}^{Eis} (4)$.
\end{theorem}

We will apply our results to $\mathcal{H}$. To this end, we first determine the action of $w_4$ on $\mathcal H.$ 
%Note that as an element of the plus-space, $\mathcal{H}$ is fixed by the Kohnen operator $U_4 w_4$, it is not fixed by the $w_4$ map needed here.
We define a version of the non-holomorphic Eisenstein series of weight
$3/2$ for $\Gamma_0(4)$ at the cusp $0.$ For $\Re(s)>1/4$ and $\tau
\in \mathbb H$, we set
$$E_{\frac{3}{2},s}(\tau):=\sum_{\substack{m>0 \\ (m, 2n)=1}} \frac{\left ( \frac{n}{m}\right )\left ( \frac{-1}{m}\right )^{\frac12}}{(mz+n)^{\frac32}|mz+n|^{2s}}.$$
This has an analytic continuation to the entire $s$-plane 
and its value at $s=0$ is a function $E_{\frac{3}{2}, 0}(\tau)$ which is not holomorphic in $\tau$ but satisfies 
the transformation equation of a modular form of weight $3/2$ (\cite{HZ}).
%and is $\Gamma_0(4)$-invariant in weight $\frac32$. 
With this notation we have the following. 
\begin{lemma}\label{H} With the definition of the action of $w_4$ given by \eqref{eq:fricke}, we have
$$\mathcal{H}|w_4=\frac{1+i}{\sqrt{8}}\mathcal{H}-\frac{1}{32\sqrt{2}}E_{\frac{3}{2}, 0}.$$
\end{lemma}
\begin{proof}
Set $$F_{\frac{3}{2}, s}(\tau)=\tau^{-\frac{3}{2}}|\tau|^{-2s}E_{\frac{3}{2}, s}(-1/(4\tau)).$$
In \cite{HZ} (Sect. 2.2), it was shown that 
$$\mathcal{H}(\tau)=\frac{-1}{96}\left ( (1-i)E_{\frac{3}{2}, 0}(\tau)-i F_{\frac{3}{2}, 0}(\tau) \right ).$$
Then, for all $\tau \in \mathbb H$, we have
\begin{align*}\label{explinv}
 (\mathcal{H}|w_4)(\tau) 
 &= 4^{\frac{3}{4}} (4 \tau)^{-\frac{3}{2}} \mathcal{H}(-1/4 \tau)\\
 &= 4^{\frac{3}{4}} (4 \tau)^{-\frac{3}{2}}\left(\frac{-1}{96}((1-i)E_{\frac{3}{2},0}(-1/4\tau) - iF_{\frac{3}{2},0}(-1/4\tau))\right)\\
 %&=4^{\frac{3}{4}} (4 \tau)^{-\frac{3}{2}}\left(\frac{-1}{96}((1-i)\tau^{3/2}F_{\frac{3}{2},0}(\tau) - i(-1/4\tau)^{-\frac{3}{2}}E_{\frac{3}{2}}(\tau))\right)
 %&= - \frac{1}{96} 4^{-\frac{3}{4}}   \left( (1-i) \tau^{\frac{3}{2}} \tau^{-\frac{3}{2}} F_{\frac{3}{2},0}(\tau) - i \tau^{-\frac{3}{2}} (-1/4\tau)^{-\frac{3}{2}}  E_{\frac{3}{2},0} (\tau) \right) \nonumber\\ 
&=  \frac{i (1-i)}{96 \sqrt{8}}    \left( i  F_{\frac{3}{2},0}(\tau) + \frac{1}{1-i} \tau^{-\frac{3}{2}} (-1/4\tau)^{-\frac{3}{2}}  E_{\frac{3}{2},0} (\tau) \right) \nonumber\\
%&= \frac{i+1}{96 \sqrt{8}}   \left(i  F_{\frac{3}{2},0}(\tau) + (1+i)4 \tau^{-\frac{3}{2}} (-1/\tau)^{-\frac{3}{2}}  E_{\frac{3}{2},0} (\tau) \right) \nonumber\\
%&=  \frac{i+1}{96 \sqrt{8}}   \left( i  F_{\frac{3}{2},0}(\tau) + (1+i)4i E_{\frac{3}{2},0} (\tau) \right)= \frac{i+1}{96 \sqrt{8}}    \left( i  F_{\frac{3}{2},0}(\tau) - (1-i)4 E_{\frac{3}{2},0} (\tau) \right) \nonumber \\
&=
\frac{i+1}{\sqrt{8}} \frac{-1}{96}   \left((1-i) E_{\frac{3}{2},0} (\tau)- i  F_{\frac{3}{2},0}(\tau) \right)-\frac{3(1+i)(1-i)}{96\sqrt{8}}E_{\frac{3}{2}, 0}(\tau)
\end{align*}
From this we deduce the lemma.
\end{proof}
Because of Lemma \ref{H}, to apply Theorem \ref{thm:summationformula} we need information about the Fourier coefficients of $E_{\frac{3}{2}, 0}.$ These are given in \cite{HZ, Zens}. Specifically, in \cite{HZ}, Sect. 2.2, $E_{\frac{3}{2}, s}(\tau)$ is decomposed
as \begin{equation}
    \label{FEE32}
E_{\frac{3}{2}, s}(\tau)=\sum_{n \in \mathbb Z} E(-n, 1+2s)\alpha_n(s, v)q^n
\end{equation}
where
$$\alpha_n(s, v):=v^{-\frac{1}{2}-2s}e^{2 \pi n v} \int_{-\infty}^{\infty} \frac{e^{-2 \pi i n vt}dt}{(t+i)^{\frac{3}{2}}(1+t^2)^s}$$
and where $E(-n, 1+2s)$ was computed in (61) of \cite{Zens} to equal
$$E(-n, 1+2s)=\frac{L(\chi_{-n}, 1+2s)(1-2^{-1-2s}\chi_{-n}(2))}{\zeta(2+4s)(1-2^{-2-4s})}r_n^{-1-4s}\sum_{m|r_n}\mu(m)\chi_{-n}(m)m^{2s}\sigma_{4s+1}(r_n/m).$$ Here $r_n^2$ is the largest odd square dividing $n$ and $\chi_{-n}$ is the character of $\mathbb Q(\sqrt{-n}).$ 
By \cite{HZ} (pg. 95) we have that $\alpha_n(0, v)=-4 \pi (1+i)n^{1/2},$ if $n>0$ and $0$ otherwise. We also have that
\begin{equation}
\label{alpha'}\alpha'_n(0, v)=- \pi (1+i) v^{-\frac{1}{2}}\int_1^{\infty}e^{-4 \pi |n| vt} t^{-\frac{3}{2}}dt \qquad \text{for $n \le 0$}.
\end{equation}
From these identities we deduce, first, that, the coefficient of $q^n$ in \eqref{FEE32} equals
$$\frac{-16 (1+i) \sqrt{n}}{\pi r_n}L(\chi_{-n}, 1)(2-\chi_{-n}(2))\sum_{m|r_n}\mu(m)\chi_{-n}(m)\sigma_{1}(r_n/m) \quad \text{if $n>0$}$$
and $0$ if $n<0$ and $-n \neq \Box$.  

If $n \le 0$ and $-n=\ell^2$, ($\ell \ge 0$), then $\chi_{-n}$ is the trivial character and thus $L(\chi_{-n}, 1+2s)=\zeta(1+2s)$ has a pole at $s=0$ with residue $1/2.$ Therefore, $E(-n, 1+2s)$ has a pole at $s=0$ with residue 
$$\frac{1}{2\zeta(2)}\left ( \frac{1-2^{-1}}{1-2^{-2}} \right ) r_n^{-1} \sum_{m|r_n}\mu(m) \sigma_1(r_n/m)=\frac{2}{\pi^2}.$$
Here we've used that, by M\"obius Inversion, $\sum_{m|r_n}\mu(m) \sigma_1(r_n/m)=r_n.$ This, with \eqref{alpha'}, implies that, if $n=-\ell^2 \le 0$ the value of $E(-n, 1+2s) \alpha_n(s, v)$ at $s=0$ is
$$\text{Res}_{s=0}(E(-n, 1+2s)) \alpha'_{n}(0, v)= \begin{cases}
\frac{-4}{\sqrt{\pi}}(1+i) \ell \Gamma \left ( \frac{-1}{2}, 4 \pi \ell^2 v \right ) \quad \text{if $n=-\ell^2<0$},\\
\frac{-4(1+i)}{\pi \sqrt{v}} \quad \text{if $n=0.$}
\end{cases}
$$
Hence we can apply Theorem \ref{thm:summationformula}, with the following choices of $(a^{\pm}(n)), (b^{\pm}(n))$. For $n>0$ and the notation 
\begin{equation}\label{Tn}
T_n:=\frac{1}{r_n}\sum_{m|r_n}\mu(m)\chi_n(m)\sigma_{1}(r_n/m).    
\end{equation} we consider
\begin{align}\label{as}
&a^+(n) = H(n) \\
& a^-(n) = \frac{\sqrt{n}}{4\sqrt{\pi}} \text{ if $n = \Box$ and $0$ otherwise} \nonumber \\
&b^+(n) = \frac{1+i}{\sqrt{8}}H(n)+\frac{(1+i) \sqrt{n}}{2\sqrt{2}\pi}L(\chi_{-n}, 1)(2-\chi_{-n}(2))T_n \nonumber \\
&b^-(n) = \frac{1+i}{8\sqrt{2\pi}}\sqrt{n} -\frac{1}{32\sqrt{2}}\frac{-4}{\sqrt{\pi}}(1+i) \sqrt{n} =\frac{(1+i)\sqrt{n}}{4\sqrt{2\pi}} \quad \text{if $n=\Box$ and $0$ otherwise} \nonumber \\
& a^+(0)=-1/12 \nonumber \\
&a^{-}(0) = \frac{1}{8\pi} \nonumber \\
&b^+(0) = -\frac{1+i}{12\sqrt{8}} \nonumber \\ 
&b^-(0) = \frac{1+i}{8\sqrt{8}\pi} -\frac{1}{32\sqrt{2}}\frac{-4(1+i)}{\pi} = \frac{3(1+i)}{16\pi\sqrt{2}},\nonumber
\end{align}

We have $\mu_f^+=\mu_g^+= \frac{1}{2} + \epsilon$ by the growth of Hurwitz class numbers \cite[Lemma 7.2]{Walker-2024} and by $L(\chi_{-n}, 1) \ll \log n.$ From \eqref{as} we get $\mu_f^- =\mu_g^- = 1/2$. 

With these remarks, upon an application of Theorem \ref{thm:extendedsummation}, we deduce the following. 
\begin{theorem}\label{Prof82} Let $H(n)$ denote the Hurwitz class number, $\chi_{-n}$ the character associated with $\mathbb Q(\sqrt{-n})$ and $T_n$ (resp. $g_{\rho}(n,x)$) be as given in \eqref{Tn} (resp. Theorem \eqref{thm:perrongen}). Then, for $\rho>0$. 
\begin{align}\label{eq:Prof82}
&\frac{1}{\Gamma(\rho+1)}\sum_{n\leq x}H(n)(x-n)^{\rho} +\frac{x^{\rho}}{8 \pi^{\frac{3}{2}} i } \sum_{n \le \sqrt{x}} n g_{\rho} \left(n^2, x \right) \nonumber
\\&-x^{\rho} \left ( \frac{1}{12 \Gamma(\rho+1)}+\frac{3xi^{\frac{3}{2}}(1+i)}{16 \Gamma(\rho+2)}-\frac{\sqrt{2 \pi x}}{8 \pi \Gamma(\rho+3/2)}-\frac{(1+i)(\pi i x)^{\frac{3}{2}}}{12\sqrt{8} \Gamma(\rho+5/2)}\right ) \nonumber \\
&=  \frac{-i^{\frac{3}{2}} (1+i)x^{\frac{2\rho+3}{4}}}{2 \sqrt{2}\pi^{\rho}}\sum_{n \ge 1} \frac{H(n)+\sqrt{n}\pi^{-1}L(\chi_{-n}, 1)(2-\chi_{-n}(2))T_n}{n^{\frac{2\rho+3}{2}}}J_{\rho+\frac{3}{2}}\left(2\pi\sqrt{nx}\right) \nonumber 
\\
&-
\frac{i^{\frac{3}{2}}(1+i)x^{\frac{\rho+1}{2}}}{4 \sqrt{2}\pi^{\rho+1}}\sum_{n \ge 1}\frac{1}{n^{\rho+1}}
\int_{0}^{1/2}\frac{u^{\frac{\rho}{2}}}{(1-u)^{\frac{1+\rho}{2}}
}J_{\rho+1}\left(2\pi n \sqrt{\frac{x(1-u)}{u}}\right)du.
\end{align}
\end{theorem}

Combining Theorem \ref{th:SummationAsymptotic} we deduce the following.
\begin{theorem}\label{applHur} %Let $H(n)$ denote the Hurwitz class number.
%and let $g_{\rho}(n,x)$ be as given in theorem \eqref{thm:perrongen}. 
For $\rho>1$ and any $\epsilon >0$, we have
\begin{align}\label{eq:applHur}
&\frac{1}{\Gamma(\rho+1)}\sum_{n\leq x}H(n)(x-n)^{\rho} 
%+\frac{x^{\rho}}{8 \pi^{\frac{3}{2}} i } \sum_{n \le \sqrt{x}} n g_{\rho} \left(n^2, x \right)
=\frac{ \pi^{\frac{3}{2}}}{24 \Gamma \left (\rho + \frac{5}{2} \right )}
x^{\rho + \frac{3}{2}} + O(x^{\rho+1 + \epsilon}).
\end{align}    
\end{theorem}
\begin{proof} From Theorem \ref{th:SummationAsymptotic} we have
%The proof is the same as Theorem \ref{th:SummationAsymptotic} with $f=\mathcal H,$ but using Lemma \ref{lem:extendrho} we can improve the bound $\rho > \rho_0$ to $\rho > 1$:
\begin{align}
&\frac{1}{\Gamma(\rho+1)}\sum_{n\leq x}H(n)(x-n)^{\rho} -\frac{i x^{\rho}}{8 \pi^{\frac{3}{2}}  } \sum_{n \le \sqrt{x}} n g_{\rho} \left(n^2, x \right)=\frac{ \pi^{\frac{3}{2}}}{24 \Gamma \left (\rho + \frac{5}{2} \right )} x^{\rho + \frac{3}{2}} + O(x^{\rho+1}).
\end{align}
By Lemma \ref{suppl}, the second sum on the left hand side is $O(x^{\rho + 1 + \epsilon})$ (where we are taking $\alpha = 1 + \epsilon$ in that lemma and using that $L_f^-(s)$ is a multiple of $ \sum_{n=1}^{\infty} \frac{n}{n^{2s}}.$
\end{proof}

\subsubsection{Numerical results}
We tested the asymptotic given by Theorem \ref{applHur} for various $\rho$. In the range of convergence, the main term exhibits the expected asymptotic behavior, as illustrated in Figure \ref{fig:varyrho}. 

Furthermore, in the region of absolute convergence $\rho>3/2$ and in the region of conditional convergence, $\rho \in (1,3/2)$, we tested the asymptotic behavior for the first error term as predicted by Theorem \ref{th:SummationAsymptotic}. The ratio of $\sum_{n\leq x} H(n)(x-n)^{\rho} - \frac{ \pi^{\frac{3}{2}}\Gamma(\rho+1)}{24 \Gamma \left (\rho + \frac{5}{2} \right )}x^{\rho + 3/2}$ and the first error term $\frac{x^{\rho + 1}3(1+i)i^{3/2}}{16\Gamma(\rho + 2)}$ for $\rho = 1.5, 2, 5, 10$ are in Figure \ref{fig:error1}

\begin{figure}[ht!]
    \centering
    \includegraphics[width=12cm]{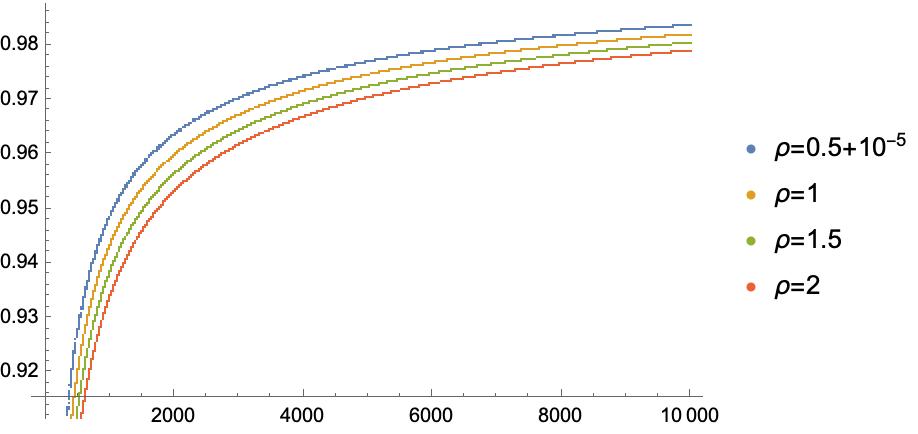}
    \caption{Ratio of $\sum_{n\leq x} H(n)(x-n)^{\rho} $  and the leading term $\frac{ \pi^{\frac{3}{2}}\Gamma(\rho+1)}{24 \Gamma \left (\rho + \frac{5}{2} \right )}x^{\rho + 3/2}$
for values of $\rho = 0.5 + 10^{-5},1,1.5,2$}
    \label{fig:varyrho}
\end{figure}

\begin{figure}
    \centering
    \includegraphics[width=11cm]{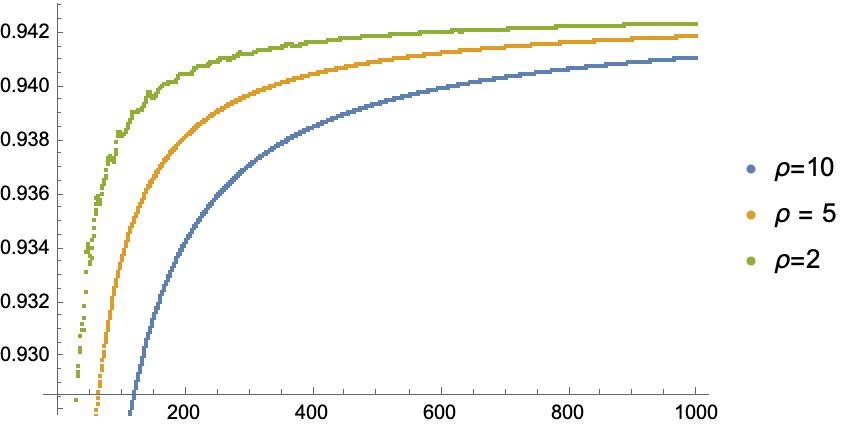}
    \caption{The ratio of $\sum_{n\leq x} H(n)(x-n)^{\rho} - \frac{ \pi^{\frac{3}{2}}\Gamma(\rho+1)}{24 \Gamma \left (\rho + \frac{5}{2} \right )}x^{\rho + 3/2}$ and the first error term $\frac{x^{\rho + 1}3(1+i)i^{3/2}}{16\Gamma(\rho + 2)}$}
    \label{fig:error1}
\end{figure}

We can extend these numerical computations to negative values of $\rho$. We observe more erratic behavior as $\rho$ moves away from the above range, but the asymptotic behavior of the main term appears to be similar. See Figure \ref{fig2}. 

\begin{figure}[ht!]
\centering
\begin{subfigure}{12cm}
\includegraphics[width=12cm,trim={1.5cm 0 0 0},clip]{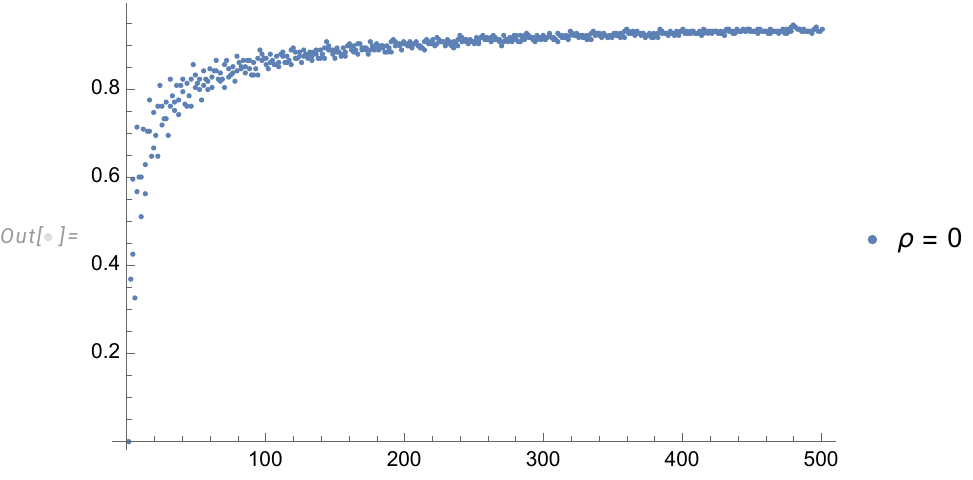}    
\end{subfigure}  
\begin{subfigure}{12cm}
\includegraphics[width=12cm,trim={1.5cm 0 0 0},clip]{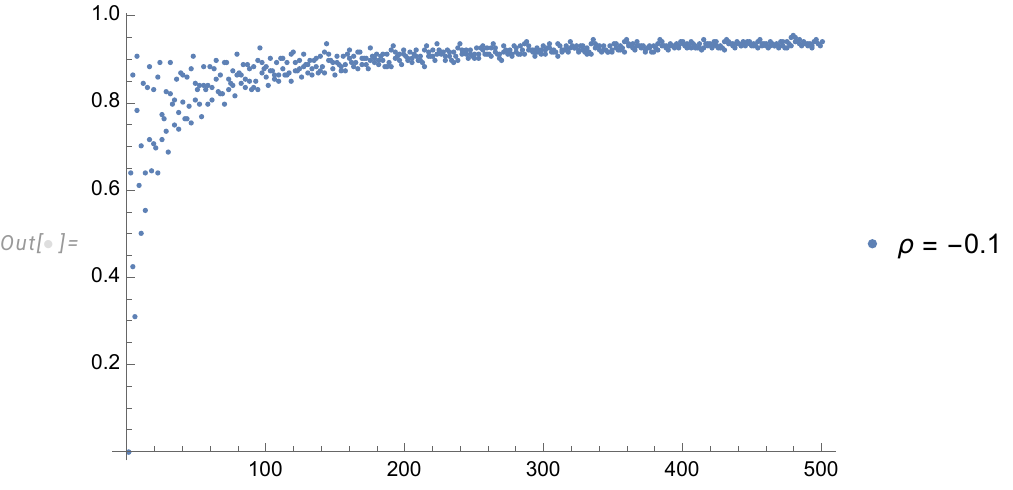}
\end{subfigure}
\begin{subfigure}{12cm}
    \includegraphics[width=12cm,trim={1.5cm 0 0 0},clip]{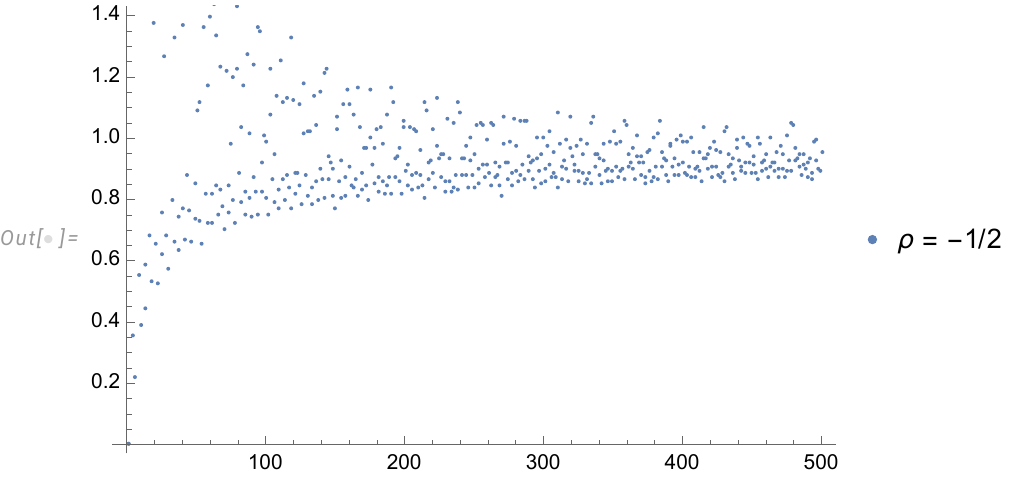}
\end{subfigure}
\caption{$\sum_{n\leq x} H(n)(x-n)^{\rho} $  and the potential leading term $\frac{ \pi^{\frac{3}{2}}\Gamma(\rho+1)}{24 \Gamma \left (\rho + \frac{5}{2} \right )}x^{\rho + 3/2}$
for values of $\rho = 0,-0.1,-0.5$}
\label{fig2}
\end{figure}

In Figure \ref{fig2} we notice that, although the asymptotic behavior for negative $\rho$ is generally similar to that for positive $\rho$, the pattern of oscillations is increasingly different from the case of $\rho>0$. This indicates that the range of summation formula for $H(n)$ that we managed to extend to $\rho>0$ using results from \cite{ChandNaras61} may be the largest possible. Equally, the behavior for negative $\rho$ exhibited in Fig. \ref{fig2} may be considered as evidence that an analogue of Theorem IV of \cite{ChandNaras61} exists for our Theorem \ref{thm:summationformula}. Theorem IV of \cite{ChandNaras61} replaces the sums $F_{\nu}$ with their Cesaro means (see \cite[13.3]{Titch}) of increasing order for smaller $\rho$. Extending their Theorem IV to our setting and bounding the resulting Cesaro means could provide a future approach to extending our summation formula and asymptotic to arbitrarily small $\rho$.

\subsection{Eisenstein series of negative half-integer weight.}\label{neg1/2}
We can also apply our theorem to non-holomorphic Cohen–Eisenstein series of {\it negative} half-integer weight. More precisely, we will apply it to the harmonic lift of a certain special value of the non-holomorphic Cohen–Eisenstein series. That lift was considered in \cite{Mizuno}, \cite{Wagner} and its coefficients involve Dirichlet $L$-values. The computations of the coefficients have been carried out in \cite{Sh75} and \cite{GH}, but we follow the presentation of \cite{IS}.

For odd $k \ge 5$, $\tau=x+iy \in \mathbb H$ and $\operatorname{Re}(s)> 2+\frac{k}{2}$ consider the non-holomorphic Eisenstein series:
$$
E(k, s, \tau) := y^{s/2} \sum_{\substack{\text{odd} \, \,  d \ge 1 \\ c \in \mathbb Z}} \left( \frac{4c}{d} \right) \varepsilon_d^{-k} (4c\tau + d)^{k/2} |4c\tau + d|^{-s}
$$
and set 
\begin{equation}\label{P} P(\tau):=E(\tau)-2^{1-\frac{k}{2}}\cos \left ( \frac{\pi k}{4} \right ) E\left (\frac{-1}{4 \tau}\right )(-2i \tau)^{\frac{k-4}{2}},
\end{equation}
where $E(\tau):=E(k-4, k-2, \tau).$ They are both weight $2-\frac{k}{2}$ harmonic Maass forms of polynomial growth for $\Gamma_0(4).$ To describe its Fourier expansion we introduce some notation. 
We first decompose any $n \in \mathbb{Z}$ as $n=t m^2$ where $t$ is square-free and let $\chi_n$ be the character given by 
$$\chi_n(a)=\left ( \frac{(-1)^{\frac{k+1}{2}}4t}{a} \right ) \qquad \text{for $(a, 4t)=1$.}$$
Further set
$$T_{n, k}:=\sum_{\substack{\text{odd} \, a, b>0 \\ab|m}}\mu(a) \chi_n(a) a^{\frac{1-k}{2}}b^{2-k} \qquad \text{and}$$
$$A_k(n):=\frac{1+i^{-k}}{2^{k}}+\sum_{j=2}^{\infty} \sum_{\nu=1}^2(1+(-1)^{\nu}i^{-k}) 2^{-\frac{kj}{2}} S_{\nu}(j, n) \quad \text{for} \, \, S_{\nu}(j, n):= \sum_{\ell=1}^{2^j}\left ( \frac{(-1)^{\nu}2^j}{\ell}\right )e^{\frac{2 \pi i n \ell}{2^j}}.$$

This is a finite sum computed explicitly in pgs 276-277 of \cite{IS}. In particular, $S_{\nu}(j, n)$ is $0$ if $j>\text{ord}_2(n)+3$ and, otherwise,  $|S_{\nu}(j, n)| \le 2^j.$ Hence $|A_k(n)| \le 1+4 \sum_{j \le \text{ord}_2(n)+2}2^{(1-\frac{k}{2})j}=O(1).$
Finally, we define
\begin{equation*}
    \tau_n \left (y, 1, \frac{k}{2}-1 \right ):=\int_{-\infty}^{\infty}e^{-2 \pi i n x}\frac{(x-iy)^{1-\frac{k}{2}}}{x+iy}dx= \begin{cases} (-2i)^{2-\frac{k}{2}} \pi y^{1-\frac{k}{2}}e^{-2 \pi n y} \qquad \text{if $n>0$}, \\
(-2i)^{2-\frac{k}{2}} \pi y^{1-\frac{k}{2}}e^{-2 \pi n y} \frac{\Gamma \left ( \frac{k}{2}-1, -4 \pi n y \right ) }{\Gamma \left ( \frac{k}{2}-1\right )} \quad \text{if $n<0$}, \\
(-2i)^{2-\frac{k}{2}}y^{1-\frac{k}{2}}\pi \qquad \text{if $n=0.$} 
\end{cases}
\end{equation*}
The final expressions in the last identity were derived 
%by the corresponding expressions in pg. 275 of \cite{IS} combined 
with (13.18.2) and (13.18.5) of \cite{NIST:DLMF}.
Then, by \cite[Proposition 1]{Sh75}, we have
\begin{equation}  E\left (\frac{-1}{4 \tau}\right )(-2i \tau)^{\frac{k-4}{2}}=-\sum_{n \in \mathbb Z} \alpha(n) (2i)^{-\frac{k}{2}}y^{\frac{k}{2}-1}e^{2 \pi i n x}\tau_n \left (y, 1, \frac{k}{2}-1 \right ),
    \label{ISFourier}
\end{equation}
where
$$\alpha(n)=\frac{L \left (\chi_n, \frac{k-1}{2} \right )(1-\chi_n(2)2^{\frac{1-k}{2}})}{\zeta(k-1)(1-2^{1-k})}T_{n, k},$$
and, from pg 276 of \cite{IS},
$$P(\tau)=y^{\frac{k}{2}-1}+y^{\frac{k}{2}-1}\sum_{n \in \mathbb{Z}}\alpha(n)A_k(n) e^{2 \pi i n x} \tau_n \left (y, 1, \frac{k}{2}-1 \right ).$$

To apply our summation formula to $P$, we also need the Fourier expansion of $P|w_4.$ We start with the following.
\begin{lemma} With the definition of the action of $w_4$ given by \eqref{eq:fricke} (with $k-4$ in place of $k$), we have
$$P|w_4=-2^{1-\frac{k}{2}}\cos \left ( \frac{\pi k}{4}\right )P+\left (1-2^{2-k} \cos^2 \left ( \frac{\pi k}{4}\right ) \right ) E \left ( \frac{-1}{4 \tau}\right ) (-2i \tau)^{\frac{k-4}{2}}.$$
\end{lemma}
\begin{proof} Upon an application of \eqref{eq:fricke} on \eqref{P} we obtain
\begin{multline*}
    (P|w_4)(\tau)=E(-1/4\tau) (-2i\tau)^{(k-4)/2}-2^{1-k/2}\cos \left ( \frac{\pi k}{4}\right ) E(\tau)=\\
    -2^{1-\frac{k}{2}}\cos \left ( \frac{\pi k}{4}\right ) \left (P(\tau)+ \left ( 2^{1-k/2}\cos \left ( \frac{\pi k}{4}\right )-2^{k/2-1}\cos^{-1} \left ( \frac{\pi k}{4}\right )\right ) E(-1/4\tau)(-2i\tau)^{(k-4)/2} \right ).
\end{multline*}
From this we deduce the lemma.
\end{proof}
%\begin{equation}   E\left (\frac{-1}{4\tau} \right)(-2i \tau)^{\frac{k}{2}}=2^{4-2k}\pi \left ( \sum_{n=0}^{\infty} \alpha_n q^n+\sum_{n=1}^{\infty} \frac{\alpha_{-n}}{\Gamma \left ( \frac{k}{2}-1\right )} \Gamma \left (\frac{k}{2}-1, 4 \pi ny \right )q^{-n} \right ).    \end{equation}
With these remarks, we can now describe the sequences $(a^{\pm}(n))$ and $(b^{\pm}(n))$ associated with $P$ in Theorem \ref{thm:summationformula}. For  $n>0$,
\begin{align}
    \label{apm}
&a^+(n)=(-2i)^{2-\frac{k}{2}} \pi \frac{L \left (\chi_n, \frac{k-1}{2} \right )(1-\chi_n(2)2^{\frac{1-k}{2}})}{\zeta(k-1)(1-2^{1-k})}T_{n, k}A_k(n) \\
 & a^-(n)=(-2i)^{2-\frac{k}{2}} \pi \frac{L \left (\chi_{-n}, \frac{k-1}{2} \right )(1-\chi_{-n}(2)2^{\frac{1-k}{2}})}{\zeta(k-1)(1-2^{1-k})\Gamma \left (\frac{k}{2}-1 \right )}T_{-n, k}A_k(-n)
 \nonumber
\\
&b^+(n)=\frac{\pi L \left (\chi_n, \frac{k-1}{2} \right )(1-\chi_n(2)2^{\frac{1-k}{2}})}{(-2i)^{\frac{k}{2}-2}\zeta(k-1)(1-2^{1-k})}T_{n, k} \left ( \frac{2^{2-k} \cos^2 \left ( \frac{\pi k}{4}\right )-1}{(2i)^{\frac{k}{2}}}-2^{1-\frac{k}{2}}\cos \left ( \frac{\pi k}{4}\right )A_k(n) \right ) \nonumber\\ 
 & b^-(n)=\frac{\pi L \left (\chi_{-n}, \frac{k-1}{2} \right )(1-\chi_{-n}(2)2^{\frac{1-k}{2}})T_{-n, k}}{(-2i)^{\frac{k}{2}-2}\zeta(k-1)(1-2^{1-k})\Gamma \left (\frac{k}{2}-1 \right )} \left ( \frac{2^{2-k} \cos^2 \left ( \frac{\pi k}{4}\right )-1}{(2i)^{\frac{k}{2}}}-2^{1-\frac{k}{2}}\cos \left ( \frac{\pi k}{4}\right )A_k(-n) \right ) \nonumber \\
& a^+(0)=2^{2-\frac{3k}{2}}i^{\frac{k}{2}-2}(1+i^{-k}) \pi \frac{\zeta(k-2)(1-2^{2-k})}{\zeta(k-1)(1-2^{1-k})}, \nonumber \\
& a^-(0)=1, \nonumber\\
%&b^{+}(0)=\frac{(1+i^{-k})\pi}{2^k(-2i)^{\frac{k}{2}-2}}\left ( 2^{2-\frac{3k}{2}} i^{-\frac{k}{2}}\cos^2 \left ( \frac{\pi k}{4}\right )-2^{1-\frac{k}{2}}\cos \left ( \frac{\pi k}{4}\right )-(2i)^{-\frac{k}{2}}\right ) \frac{\zeta(k-2)(1-2^{2-k})}{\zeta(k-1)(1-2^{1-k})} \nonumber \\
&b^{+}(0)=\left ( 2^{2-k} \left (1-2^{2-k}\cos^2 \left ( \frac{\pi k}{4}\right ) \right )+2^{3-2k}\cos \left ( \frac{\pi k}{4}\right )(1+i^{-k}) i^{\frac{k}{2}} \right ) \pi \frac{\zeta(k-2)(1-2^{2-k})}{\zeta(k-1)(1-2^{1-k})} \nonumber \\
&b^-(0)=-2^{1-\frac{k}{2}}\cos \left ( \frac{\pi k}{4}\right ). \nonumber
\end{align}
Each $a^{\pm}(n)$, $b^{\pm}(n)$ is a product of $L(\chi_n, (k-1)/2)$ (bounded, since $k \ge 5$) with terms $O(A_k(n))$ which, as mentioned above, are bounded. Therefore, $\mu_P^{\pm}=\mu_{P|w_4}^{\pm}=\varepsilon$ and hence we can state the following application of Theorem \ref{thm:summationformula}:
\begin{theorem}
\label{ShimEis} Let $a^{\pm}(n), b^{\pm}(m)$ be given as in \eqref{apm}. Then, for $\rho>k-\frac{3}{2}$, 
\begin{align}
&\frac{1}{\Gamma(\rho+1)}\sum_{n\leq x} a^+(n)(x-n)^{\rho} +\frac{x^{\rho}}{2 \pi i } \sum_{n \le x} a^-(n) g_{\rho}\left(n, x \right) - Q_{\rho}(x) \\
%\\&-x^{\rho} \left ( \frac{1}{12 \Gamma(\rho+1)}+\frac{3xi^{\frac{3}{2}}(1+i)}{16 \Gamma(\rho+2)}-\frac{\sqrt{2 \pi x}}{8 \pi \Gamma(\rho+3/2)}-\frac{(1+i)(\pi i x)^{\frac{3}{2}}}{12\sqrt{8} \Gamma(\rho+5/2)}\right ) \nonumber \\
&= i^{\frac{-k}{2}} x^{\frac{\rho}{2}+\frac{4-k}{4}} \left ( \frac{\sqrt{N}}{2 \pi}\right )^{\rho}\sum_{n \ge 1} \frac{b^+(n)}{n^{\frac{2\rho+4-k}{4}}}J_{\rho+\frac{4-k}{2}}\left(2\pi\sqrt{nx}\right) \nonumber 
\\
&+i^{\frac{-k}{2}}x^{\frac{\rho}{2}+\frac{1}{2}}\left ( \frac{\sqrt{N}}{2 \pi}\right )^{\rho+\frac{4-k}{2}-1}\sum_{n \ge 1}\frac{b^-(n)}{n^{\frac{\rho-k+3}{2}}}\int_{0}^{1/2}\frac{u^{\frac{\rho-k+1}{2}}}{(1-u)^{\frac{1+\rho}{2}}
}J_{\rho+1}\left(2\pi n \sqrt{\frac{x(1-u)}{u}}\right)du,
\end{align}
where $g_\rho(n, x)$ and $Q_{\rho}(x)$ are given by \eqref{Qrho} and Theorem \eqref{thm:perrongen} respectively (with $k$ replaced by $\frac{4-k}{2}$).
    \end{theorem}
The asymptotics are described by the following application of Theorem \ref{th:SummationAsymptotic}.
\begin{theorem}
\label{th:AsymShimEis}
Assume all the notation of Proposition \ref{ShimEis}. For $\rho>k-\frac{3}{2}$, we have:
\begin{equation}
\label{eq:AsymShimEis}
\frac{1}{\Gamma(\rho+1)}\sum_{n\leq x}a^{+}(n)(x-n)^{\rho} +\frac{x^{\rho}}{2 \pi i } \sum_{n \le x} a^-(n) g_{\rho} \left(n, x \right)=\frac{2^{2-k} i^{-\frac{k}{2}} \pi x^{\rho+1} \cos \left ( \frac{\pi k}{4}\right )}{2\Gamma(\rho + 2)} + O(x^{\rho}).
\end{equation}
    \end{theorem}
    
\section{Converse Theorem}\label{Converse}
In this section, we show the analogue of the second part of Lemma 5 of \cite{ChandNaras61}, namely that the summation formula of Theorem \ref{thm:summationformula} implies the functional equation \eqref{eq:fcnleq}. We first prove the following.
\begin{theorem}\label{converse} Let $k \in \frac{1}{2}\mathbb{Z}$ with $k<2$. Suppose that $(a^{\pm}(n)), (b^{\pm}(n))$ are such that $a^{\pm}(n) = O(n^{\mu_f^{\pm}})$ and $b^{\pm}(n) = O(n^{\mu_g^{\pm}})$, for some $\mu_f^{\pm}, \mu_g^{\pm} \ge 0$. 
%Set $\alpha=\max(\mu_f^{-}+1, k-1)+\varepsilon$. 
If \eqref{eq:sumformula} holds for all $$\rho>\max \left (1-k, 
%-1-\alpha, 
\mu_g^+-\frac{k}{2}+\frac{3}{4}, 2\mu_g^--2k+\frac{5}{2} \right ),$$ 
 then we have
\begin{align}
&\sum_{n=1}^{\infty}a^{+}(n)e^{-ny}+\sum_{n=1}^{\infty}a^-(n)e^{ny}\Gamma(1-k, 2ny)\nonumber\\
    &\qquad+a^{+}(0)-\frac{2 \pi N^{\frac{k}{2}-1}  i^k}{y}b^-(0)+\frac{ ( 2 \pi )^{k-1}}{ y^{k-1}} a^-(0) -\frac{i^k (2 \pi)^k}{N^{\frac{k}{2}} y^{k}}b^{+}(0)\nonumber\\
    &=-\left(\frac{2\pi i}{\sqrt{N}y}\right)^k\sum_{n=1}^{\infty}b^{+}(n)e^{-\frac{4n\pi^2}{Ny}}-\left(\frac{2\pi i}{\sqrt{N}y}\right)^k\sum_{n=1}^{\infty}b^-(n)e^{\frac{4n\pi^2}{Ny}}\Gamma\left(1-k,\frac{8n\pi^2}{Ny} \right).
    \end{align}
\end{theorem}
\begin{proof}
Multiply \eqref{eq:sumformula} by $y^{\rho+1}e^{-yx}$ (for $y>0$) and then integrate with respect to $x$ from $0$ to $\infty$. We rewrite the left-hand side of the resulting equality as follows:
\begin{align*}
%\label{first}
    &\int_{0}^{\infty}\left(\frac{1}{\Gamma(\rho+1)}\sum_{n\leq x}a^{+}(n)(x-n)^{\rho} +\frac{x^{\rho}}{2 \pi i } \sum_{n \le x} a^-(n) g_{\rho}\left(n, x \right)-Q_{\rho}(x)\right)y^{\rho+1}e^{-yx}dx=I_1+I_2+I_3.
    % &\frac{y^{\rho+1}}{\Gamma(\rho+1)}\int_{0}^{\infty}\sum_{n\leq x}a^{+}(n)(x-n)^{\rho}e^{-yx}dx +\frac{y^{\rho+1}}{2 \pi i } \int_{0}^{\infty}\sum_{n \le x} a^-(n) g_{\rho}\left(n, x \right)x^{\rho}e^{-yx}dx+\cdots\nonumber\\
    % &\sum_{n=1}^{\infty}a^{+}(n)e^{-ny} +\frac{y^{\rho+1}}{2 \pi i } \int_{0}^{\infty}\sum_{n \le x} a^-(n) g_{\rho}\left(n, x \right)x^{\rho}e^{-yx}dx\cdots\nonumber\\
\end{align*}
Since $\rho>1-k>-1$, we have
\begin{multline}\label{I1}
    I_1:=\int_{0}^{\infty}\frac{1}{\Gamma(\rho+1)}\sum_{n\leq x}a^{+}(n)(x-n)^{\rho}y^{\rho+1}e^{-yx}dx\\ =\frac{y^{\rho+1}}{\Gamma(\rho+1)}\sum_{n=1}^{\infty}a^{+}(n)\int_{n}^{\infty}(x-n)^{\rho}e^{-nx}dx
%    &=\frac{y^{\rho+1}}{\Gamma(\rho+1)}\sum_{n=1}^{\infty}a^{+}(n)y^{-\rho-1}\Gamma(\rho+1)e^{-ny}\nonumber\\
    =\sum_{n=1}^{\infty}a^{+}(n)e^{-ny}.
\end{multline}

By Theorem \ref{thm:perrongen}, since $\rho+k>1$, we have, for some $\alpha>\mu_f^{-}+1
%\max(\mu_f^{-}+1, k-1)
$,
\begin{align}
    I_{2}&:=y^{\rho+1}\int_{0}^{\infty}\frac{x^{\rho}}{2 \pi i } \sum_{n \le x} a^-(n) g_{\rho}\left(n, x \right)e^{-x y}dx=
 y^{\rho+1}\int_{0}^{\infty}\frac{x^{\rho}}{2 \pi i }\int_{(\alpha)}\frac{L_f^-(s)W_{1-k}(s)x^sds}{\Gamma(\rho+1+s)} \\
 &=\frac{1}{2 \pi i}\int_{(\alpha)}\frac{L_f^-(s)W_{1-k}(s)y^{\rho+1}}{\Gamma(\rho+1+s)}\int_0^{\infty}e^{-yx}x^{s+\rho}dxds\nonumber
=\frac{1}{2 \pi i}\int_{(\alpha)}L_f^-(s)W_{1-k}(s)y^{-s}ds.
\end{align}
The interchange of integration follows from Fubini's theorem because $\rho>1-k>-\alpha-1$. Therefore, if $\mathcal{M}^{-1}(h)$ denotes the inverse Mellin transform of $h$, we have 
\begin{align}\label{I2}
I_2=\frac{1}{2 \pi i}\sum_{n \ge 1}a^-(n)\int_{(\alpha)}W_{1-k}(s)(ny)^{-s}ds&=\sum_{n \ge 1}a^-(n)\mathcal{M}^{-1}\left (W_{1-k}(s) \right )(ny) \nonumber\\
&=\sum_{n=1}^{\infty}a^-(n)\Gamma(1-k, 2ny)e^{ny}
\end{align}
because, by definition, $W_{1-k}(s)=\mathcal{M}(\Gamma(1-k, 2x)e^{x})(s).$ Finally, since $\rho>-k$ and $\rho>-1$, we have 
\begin{align}\label{I3}
    I_3:&=y^{\rho+1}\int_{0}^{\infty} x^{\rho} \left( \frac{a^+(0)}{\Gamma(\rho + 1)}-\frac{2 \pi x N^{\frac{k}{2}-1} b^-(0) i^k}{\Gamma(\rho + 2)}+ \frac{ a^-(0) x^{k-1} ( 2 \pi )^{k-1}}{ \Gamma(\rho +  k)}  - \frac{b^{+}(0) i^k x^{k} (2 \pi)^k}{N^{\frac{k}{2}} \Gamma(\rho + k + 1)} \right)e^{-xy}dx\nonumber\\
    &=a^{+}(0)-\frac{2 \pi N^{\frac{k}{2}-1}  i^k}{y}b^-(0)+\frac{ ( 2 \pi )^{k-1}}{ y^{k-1}} a^-(0) -\frac{i^k (2 \pi)^k}{N^{\frac{k}{2}} y^{k}}b^{+}(0).
\end{align}

We now consider the right-hand side of the equality obtained after multiplying \eqref{eq:sumformula} by $y^{\rho+1}e^{-yx}$ and integrating from $0$ to $\infty$. Its first term is  
\begin{align}
    I_4:=y^{\rho+1}\int_{0}^{\infty}-i^k x^{(\rho + k)/2}\left(\frac{\sqrt{N}}{2\pi }\right)^{\rho}\sum_{n=1}^{\infty}\frac{b^{+}(n)}{n^{(\rho+k)/2}}J_{\rho+k}\left(4\pi\sqrt{\frac{nx}{N}}\right)e^{-xy}dx.
\end{align}
Since $\rho>-\frac{5}{2}-k$ and $\rho> \mu_g^+-\frac{k}{2}+\frac{3}{4}$, Proposition \ref{thm:besselj-asymptotics} implies that we have absolute convergence. Thus we can interchange summation and integration to deduce, 
after an application of (19) of \cite{ChandNaras61}, 
\begin{align}\label{J1}
    I_4=-\left(\frac{2\pi i}{\sqrt{N}y}\right)^k\sum_{n=1}^{\infty}b^{+}(n)e^{-\frac{4n\pi^2}{Ny}}.
\end{align}
Regarding the second term
$$I_5:=-i^ky^{\rho+1}\left(\frac{\sqrt{N}}{2\pi}\right)^{\rho+k-1}\int_{0}^{\infty} x^{\frac{\rho+1}{2}}\sum_{n=1}^{\infty} \frac{b^-(n)}{n^{\frac{\rho-1}{2}+k}}\int_{0}^{\frac{1}{2}}\frac{u^{k+\frac{\rho-3}{2}}}{(1-u)^{\frac{1+\rho}{2}}
}J_{\rho+1}\left(4\pi\sqrt{\frac{nx(1-u)}{Nu}}\right)e^{-xy}dudx,$$
we can interchange summation and integrations because $\rho$ is larger than $2\mu_g^-+\frac{5}{2}-2k,$ $\frac{1}{2}-2k$ and $-\frac{5}{2}$ to get, with (19) of \cite{ChandNaras61},
\begin{align}\label{J2}
-i^ky^{\rho+1}\left(\frac{\sqrt{N}}{2\pi}\right)^{\rho+k-1}&\sum_{n=1}^{\infty} \frac{b^-(n)}{n^{\frac{\rho-1}{2}+k}}\int_{0}^{\frac{1}{2}}\frac{u^{k+\frac{\rho-3}{2}}}{(1-u)^{\frac{1+\rho}{2}}
}\int_{0}^{\infty} x^{\frac{\rho+1}{2}}J_{\rho+1}\left(4\pi\sqrt{\frac{nx(1-u)}{Nu}}\right)e^{-xy}dxdu\nonumber\\
&=-\frac{4\pi^2}{Ny}\left(\frac{-\sqrt{N}}{2\pi i} \right)^{k}\sum_{n=1}^{\infty}\frac{b^{-}(n)}{n^{k-1}}\int_{0}^{1/2}u^{k-2}e^{-\frac{4n\pi^2}{Ny}\left(\frac{1-u}{u}\right)}du\nonumber\\
&=-\left(\frac{2\pi i}{\sqrt{N}y}\right)^k\sum_{n=1}^{\infty}b^-(n)e^{\frac{4n\pi^2}{Ny}}\Gamma\left(1-k,\frac{8n\pi^2}{Ny} \right).
\end{align}
Combining \eqref{I1}, \eqref{I2}, \eqref{I3}, \eqref{J1} and \eqref{J2}, we deduce the theorem. 
\end{proof}
\begin{corollary}
    With the terminology and assumptions of Theorem \ref{converse}, set
    $$
f(\tau) = \sum_{n=0}^{\infty} a^+(n) q^n + a^{-}(0) y^{1-k}+ \sum_{n=1}^{\infty} a^-(n)\Gamma(1-k, 4 \pi n y) q^{-n}
$$
and
\begin{align*}
g(\tau):= \sum_{n=0}^{\infty} b^+(n) q^n + b^-(0) y^{1-k} + \sum_{n=1}^{\infty} b^-(n) \Gamma(1-k, 4 \pi n y) q^{-n}.
\end{align*}
Further set $$\Lambda(f,s)=
\int_0^{\infty} \left( f(it/\sqrt{N}) - a^+(0) - \frac{a^-(0)}{N^{(1-k)/2}} t^{1-k} \right) t^{s-1} dt,$$
and let $\Lambda(g, s)$ be the corresponding function for $g$. Then, $\Lambda(f,s), \Lambda(g, s)$ are meromorphic in $\mathbb C$ and satisfy $\Lambda(f, s)=i^k\Lambda(g, k-s)$
\end{corollary}
\begin{proof}
The conclusion of Theorem \ref{converse} can be rewritten as $f(i/n\sqrt{N})=g(iu/\sqrt{N})(iu)^k.$ This and the exponential decay of $f(it/\sqrt{N}) - a^+(0) - a^-(0)N^{(k-1)/2} t^{1-k}$ and $g(it/\sqrt{N}) - b^+(0) - b^-(0)N^{(k-1)/2} t^{1-k}$ are the assumptions on which the proof of Theorem \ref{thm:ShankadharSingh} relies. Repeating that argument, we deduce the assertion of the corollary. 
\end{proof}
\begin{remark}
This theorem could be made into a full converse theorem. For that, we would need, firstly, a companion summation formula to ensure that the functions $f, g$ of the corollary satisfy $g=f|w_N$ throughout the upper half-plane, and, secondly, twisted versions of those summations, to extend the invariance under the group, as in the Weil Converse Theorem. 
\end{remark}
%%%%%%%%%%%%

\bibliographystyle{amsalpha}
\bibliography{references.bib}

\end{document}